\documentclass{scrartcl}

\usepackage[
final,
theoremdefs,
showeqnr,
]{allergy}
\usepackage{unixode}

\usepackage{helvetic}
\colorlet{seccolor}{SaddleBrown}
\addtokomafont{sectioning}{\color{DarkRed}}
\addtokomafont{title}{\color{DarkGoldenrod}}

\usetikzlibrary{arrows}

\usepackage{subfig}

\usepackage{authblk}

\newcommand*\email[1]{\href{mailto:#1}{\nolinkurl{#1}}}

\newcommand*\mytitle{Extension of Matrix Pencil Reduction to Abelian Categories}
\title{\mytitle}
\author{Olivier Verdier\thanks{\email{olivier.verdier@hvl.no}}}
\affil{Department of Computing, Mathematics and Physics\\
Western Norway University of Applied Sciences, Norway}
\affil{Department of Mathematics and Mathematical Statistics\\
  Umeå University, Umeå, Sweden}

\newcommand\obsredmat{
	\begin{scope}[yscale=-1,scale=1.2]
		\draw[step=1cm] (0,0) grid (2,2);
		\node at (.5,.5) {$\E\ored,\A\ored$};
		\node at (1.5,.5) {$\star,\star$};
		\node at (1.5,1.5) {$0, \pline$};
		\node at (.5,1.5) {$0, 0$};
		\node[above] at (.5,0) {$\Mv\ored$};
		\node[above] at (1.5, 0) {$\coim\pline$};
		\node[left] at (0,.5) {$\Vv\ored$};
		\node[left] at (0,1.5) {$\coker\E$};
	\end{scope}
}

\newcommand{\contredmat}{
	\begin{scope}[yscale=-1,scale=1.2]
		\draw[step=1cm] (0,0) grid (2,2);
		\node at (.5,.5) {$0, \rline$};
		\node at (.5,1.5) {$0,0$};
		\node at (1.5,1.5) {$\E\cred, \A\cred$};
		\node at (1.5,.5) {$\star, \star$};
		\node[above] at (.5,0) {$\ker\E$};
		\node[above] at (1.5, 0) {$\Mv\cred$};
		\node[left] at (0,.5) {$\im\rline$};
		\node[left] at (0,1.5) {$\Vv\cred$};
	\end{scope}
}
\NewDocumentCommand\ninediag{mmmO{}}{
	\begin{tikzpicture}[ineq]
		\matrix(m)[smallcommdiag]
		{ \& 0 \& 0 \& 0 \& \\
		0 \& 
		#1
		\& 0\\
		0 \& 
		#2
		\& 0 \\
		0 \& 
		#3
		\& 0 \\
		 \& 0 \& 0 \& 0 \&  \\};
		\path[->]
		(m-1-2) edge (m-2-2)
		(m-2-2) edge (m-3-2)
		(m-3-2) edge  (m-4-2)
		(m-4-2) edge (m-5-2)
		;
		\path[->]
		(m-1-3) edge (m-2-3)
		(m-2-3) edge (m-3-3)
		(m-3-3) edge  (m-4-3)
		(m-4-3) edge (m-5-3)
		;
		\path[->,
		]
		(m-1-4) edge (m-2-4)
		(m-2-4) edge (m-3-4)
		(m-3-4) edge  (m-4-4)
		(m-4-4) edge (m-5-4)
		;
		\path[->,
		]
		(m-2-1) edge (m-2-2)
		(m-2-2) edge (m-2-3)
		(m-2-3) edge  (m-2-4)
		(m-2-4) edge (m-2-5)
		;
		\path[->]
		(m-3-1) edge (m-3-2)
		(m-3-2) edge (m-3-3)
		(m-3-3) edge  (m-3-4)
		(m-3-4) edge (m-3-5)
		;
		\path[->]
		(m-4-1) edge (m-4-2)
		(m-4-2) edge (m-4-3)
		(m-4-3) edge (m-4-4)
		(m-4-4) edge (m-4-5)
		;
		#4
	\end{tikzpicture}
}

\usetikzlibrary{matrix}
\tikzset{commdiag/.style={matrix of math nodes, row sep=2.5em, column sep=2em, text height=1.5ex, text depth=0.25ex,ampersand replacement=\&},
smallcommdiag/.style={commdiag, row sep=2em, column sep=1.5em, font=\small},
exseq/.style={commdiag, column sep=2em},
diagequal/.style={double, double distance=2pt, -},
>=stealth,
ineq/.style={baseline=(current  bounding  box.center)},% chktex 36
amph/.style={very thick},
}

\usepackage{enumitem}
\newenvironment{thmenumerate}{\begin{enumerate}[label=\upshape(\roman*)]}{\end{enumerate}}

\newcommand{\picturesfolder}{./pictures}
\newcommand{\matrixfigure}[4][.6]{\begin{figure}
\begin{tikzpicture}
	\coordinate (mat) at (0,0);
	\coordinate (leg) at (9,-1);
	\node[anchor=north west] at (mat) {\includegraphics[width=#1\textwidth]{\picturesfolder/#3}};
	\node[anchor=north west] at (leg) {\includegraphics[scale=.7]{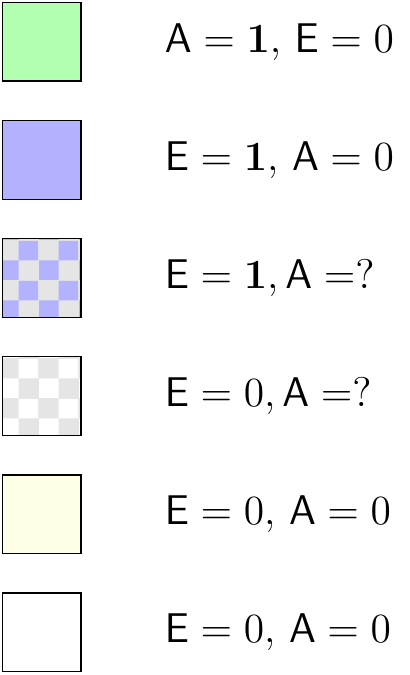}};
\end{tikzpicture}
\caption{#4}
\label{#2}
\end{figure}}

\newcommand*{\alert}[1]{\textit{\textbf{#1}}}

\newcommand*\demph[1]{\textbf{\emph{#1}}}

\newcommand*{\mat}[1]{\mathsf{#1}}

\newcommand*\NN{\mathbf{N}}
\newcommand*{\inv}{^{-1}}
\newcommand*\mul{\,}

\newcommand*\resolvent{\rho}

\DeclareMathOperator{\im}{im}

\NewDocumentCommand\diagdim{mO{}}{\textcolor{Blue}{#1}}

\newcommand*\GL[1][d]{\ensuremath{\mathsf{GL}(#1)}}
\newcommand*\RR{\mathbf{R}}

\newcommand*\Mv{M}
\newcommand*\Vv{V}

\NewDocumentCommand\ored{O{1}}{\IfValueTF{#1}{^{#1}}{}}
\NewDocumentCommand\cred{O{1}}{\IfValueTF{#1}{_{#1}}{}}
\NewDocumentCommand\rred{O{1}O{1}}{\ored[#1]\cred[#2]}

\NewDocumentCommand\rline{O{\A}}{\left\lvert{#1}\right.}
\NewDocumentCommand\pline{O{\A}}{\left.{#1}\right\rvert}
\NewDocumentCommand\prline{O{\A}}{\abs{#1}}
\newcommand{\coker}{\operatorname{coker}}
\newcommand{\coim}{\operatorname{coim}}
\NewDocumentCommand\sect{O{}m}{\bracket[#1]{#2}}

\newcommand\Uv{\Mv}
\newcommand\W{\Vv}

\newcommand\Sop[1][]{\mat{S}_{#1}}

\newcommand{\kE}{\ker\E}
\newcommand{\EU}{\im\E}

\DeclareMathOperator{\ind}{ind}

\newcommand{\E}{\mat{E}}
\newcommand{\A}{\mat{A}}
\newcommand{\Pm}{\mat{P}}
\newcommand{\Qm}{\mat{Q}}
\newcommand{\Id}{\mathbb{I}}
\newcommand\one{\mat{1}}
\newcommand\transpose{{\!\scriptstyle\mathsf{T}}} 	% Transpose

\newcommand\sys{(\E,\A)}

\newcommand{\DM}{\Delta\Mv}
\newcommand{\DV}{\Delta\Vv}

\newcommand{\CC}{\mathbf{C}}

\newcommand{\Rm}{\mat{R}}

\newcommand{\Eb}{\mat{F}}
\newcommand{\Ab}{\mat{B}}
\newcommand\vv{W}
\newcommand\mv{N}

\newcommand*\nildef[1][]{\nu_{#1}}

\newcommand{\AoneEzero}{\newcommand*{\Acol}[1]{1}\newcommand*{\Ecol}[1]{0} }
\newcommand{\EoneAzero}{\newcommand*{\Acol}[1]{0}\newcommand*{\Ecol}[1]{1} }

\begin{document}

\maketitle

\begin{abstract}
Matrix pencils, or pairs of matrices, are used in a variety of applications.
By the Kronecker decomposition Theorem, they admit a normal form.
This normal form consists of four parts, one part based on the Jordan canonical form, one part made of nilpotent matrices, and two other dual parts, which we call the observation and control part.
The goal of this paper is to show that large portions of that decomposition are still valid for pairs of morphisms of modules or abelian groups, and more generally in any abelian category.
% This gives a new perspective even in the vector space case, as we have to use radically new proof techniques to work on abelian categories.
In the vector space case, we recover the full Kronecker decomposition theorem.
The main technique is that of reduction, which extends readily to the abelian category case.
Reductions naturally arise in two flavours, which are dual to each other.
% It turns out that this technique directly extends to abelian categories.
There are a number of properties of those reductions which extend remarkably from the vector space case to abelian categories.
First, both types of reduction commute.
Second, at each step of the reduction, one can compute three sequences of invariant spaces (objects in the category), which generalize the Kronecker decomposition into nilpotent, observation and control blocks.
These sequences indicate whether the system is reduced in one direction or the other.
In the category of modules, there is also a relation between these sequences and the resolvent set of the pair of morphisms, which generalizes the regular pencil theorem.
We also indicate how this allows to define invariant subspaces in the vector space case, and study the notion of strangeness as an example.
% In particular, a pair of matrices $(\E,\A)$ may be interpreted as the differential equation $\E x' + \A x = 0$.
% Such an equation is invariant by changes of variables, or linear combination of the equations.
% This change of variables or equations is associated to a group action.
% The invariants corresponding to this group action are well known, namely the Kronecker indices and divisors.
% Similarly, for another group action corresponding to the weak equivalence, a complete set of invariants is also known, among others the strangeness.
% 
% Those invariants can be defined in a directly invariant fashion, i.e. without using a basis or an extra Euclidean structure.
% That is done via a reduction process, coming in two flavors, which produces new system out of the original one.
% We insist on the fact that those reduction procedures are defined without recourse to any choice of basis.
% The various invariants may then be defined from operators related to the repeated application of the reduction process.
% % We then show the relation between the invariants and the reduced subspace dimensions, and the relation with the regular pencil condition.
% % This is all done using invariant tools only.
% 
% Finally, we present in appendix the construction of the Kronecker canonical form, as well as the strangeness canonical form associated to weak equivalence.
\end{abstract}

\begin{description}
	\item[Keywords] Matrix Pencil, Kronecker Decomposition, Reduction, Strangeness
	\item[Mathematics Subject Classification (2010)] 15A03, 15A21, 15A22, 18Exx
\end{description}

\tableofcontents

\section{Introduction}

The purpose of this paper is to show that the Kronecker decomposition theorem for pairs of matrices (or linear matrix pencils) admits a far-reaching generalization on modules and abelian groups.

We begin by recalling the Kronecker decomposition theorem~\cite[\S\,XII.5]{Gantmacher}.
% The Kronecker decomposition theorem can be summarized as follows.%\cite{Johansson}.
Suppose we have a pair of operators $\E$, and $\A$ from $\Mv$ to $\Vv$, both finite dimensional vector spaces.
% The Kronecker theorem states that there is a special choice of basis in $\Mv$ and $\Vv$ such that the operators $\E$ and $\A$ can be decomposed as follows.

Let the rectangular {bidiagonal blocks} {$\mat{L}_k^{\E}$} and {$\mat{L}_k^{\A}$} be defined by

\newcommand{\KronLBlock}{\left.\begin{bmatrix}
\Ecol{1} &  &   &  \\
\Acol{1}  & \Ecol{1} &  & \\
& \ddots & \ddots & \\
&& \Acol{1} & \Ecol{1} \\
&& & \Acol{1}
\end{bmatrix}\quad\right\} {k+1}}

\[\mat{L}_{k}^{\E} := {\EoneAzero\KronLBlock}
,
\qquad \mat{L}_k^{\A} := {\AoneEzero\KronLBlock}
.
\]

Let the {nilpotent blocks} {$\mat{N}_k^{\E}$} and {$\mat{N}_k^{\A}$} be defined by

\newcommand{\Nilpotentblock}{\left. \begin{bmatrix}
	\Acol{1} & \Ecol{1} &  & & \\
	  & \Acol{1} & \Ecol{1} &  & \\
	  &   & \ddots & \ddots & \\
	& & & \Acol{1} & \Ecol{1}\\
		& & & & \Acol{1}
\end{bmatrix}\quad\right\} {k+1}}

\newcommand\diagEA[1]{\ensuremath{\operatorname{diag}(
\mat{N}_{k_1}^{#1},\ldots,\mat{N}_{k_m}^{#1},
\mat{L}_{k_1}^{#1},\ldots,\mat{L}_{k_p}^{#1},
\paren{\mat{L}_{k_1}^{#1}}^{\transpose},\ldots,\paren{\mat{L}_{k_q}^{#1}}^{\transpose})}}

\begin{align}
	\label{eq:def_bidinil}
	\mat{N}_k^{\E} := {\EoneAzero \Nilpotentblock}
,
\qquad \mat{N}_k^{\A} := {\AoneEzero \Nilpotentblock}
.
\end{align}

A \emph{Kronecker decomposition} of the system $(\E,\A)$ is a choice of basis of $\Mv$ and $\Vv$ such that $\E$ and $\A$ are decomposed in blocks of the same size, i.e.,
\[\E = \begin{bmatrix}
	 \Id & 0 \\ 0 & \Delta^{\E}
\end{bmatrix}, \qquad \A = \begin{bmatrix}
	\A' & 0 \\ 0 & \Delta^{\A}
\end{bmatrix}
,
\]
where $\A'$ is an arbitrary matrix which can be further reduced in Jordan canonical form, and $\Delta^{\E}$ and $\Delta^{\A}$ are in diagonal block form
\begin{subequations}
\begin{align}
\Delta^{\E} &= \diagEA{\E}
, \\
\Delta^{\A} &= \diagEA{\A}
.
\end{align} 
\end{subequations}
% where the blocks of $\E$ and $\A$ have the same size.

For a given pair $(\E,\A)$ we  can thus define, for any integer $k$, the integers $\nildef[k]$, $\beta\ored[k]$ and $\beta\cred[k]$ such that there are:
\begin{itemize}
	\item $\nildef[k]$ block of type $\mat{N}_{k}$,
	\item $\beta\ored[k]$ blocks of type $\mat{L}_k$,
	\item $\beta\cred[k]$ blocks of type $\mat{L}_k^{\transpose}$.
\end{itemize}

In order to understand how this can be generalized, we define the concept of \emph{reduction}.
The idea is that the reduction procedure can be directly generalized to abelian groups and $R$-modules.

For vector spaces, this concept was gradually developed, under various names, or no name at all, first in~\cite{Wong} for the study of regular pencils, then in~\cite[\S4]{Wilkinson} and~\cite{vanDooren} to prove the Kronecker decomposition theorem.
It is also related to the \emph{geometric reduction} of nonlinear implicit differential equations as described in~\cite{Reich} or~\cite{Rabier-Rheinboldt}. 
In the linear case, those coincide with the observation reduction~\cite{thesis}.
It is also equivalent to the algorithm of prolongation of ordinary differential equation in the formal theory of differential equations~\cite{Reid}.

In~\cite{Wilkinson}, one considers also the \emph{conjugate} of the reduction, i.e., the operation obtained by transposing both matrices, performing a reduction and transposing again.
In order to make a distinction between both operations, we call the first one ``observation'' reduction, and the latter, ``control'' reduction.
The control reduction coincides with one step of the \emph{tractability chain}~\cite{Marz}.

Both reductions also appear in the context of ``linear relations''.
For a system of operators $(\E,\A)$ defined from $\Uv$ to $\W$, there are two corresponding linear relations, which are subspaces of $\Uv\times\Uv$ and $\W\times\W$.
These are respectively called the left and right linear relation (\cite[\S~6]{BaskakovSpectral},~\cite[\S~5.6]{BaskakovRepresentation}).
As one attempts to construct semigroup operators, it is natural to study the iterates of those linear relations.
That naturally leads to iterates of observation or control reduction.

The reduction procedures come in two dual flavours, which we call observation and control reduction.
Both create  a new, smaller system, out of a pair of operators $(\E,\A)$.
The reductions are depicted in \autoref{fig:reductions}.
% ``Smaller'' is in the sense that the reduced operators $\E\ored$ and $\A\ored$ are restrictions of $\E$ and $\A$ on subspaces of $\Mv$ and $\Vv$, defined by $\Vv\ored\coloneqq\E\Mv$ and $\Mv\ored\coloneqq\A\inv\Vv\ored$.

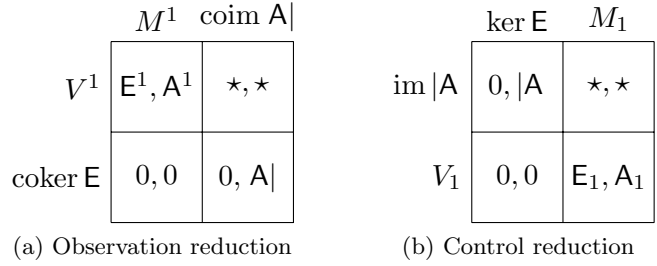
\begin{figure}
	\centering
		\subfloat[Observation reduction]{
	\begin{tikzpicture}
		\obsredmat
	\end{tikzpicture}
	}\qquad
	\subfloat[Control reduction]{
		\begin{tikzpicture}
		\contredmat	
	\end{tikzpicture}
}
	\caption{
		An illustration of the two possible reduction procedures.
		The operator $\pline$ is the operator $\A$ followed by a projection on $\coker\E$.
		The operator $\rline$ is the restriction of $\A$ on $\ker\E$.
		This suffices to define the observation-reduced system $(\E\ored,\A\ored)$ defined from $\Mv\ored$ to $\Vv\ored$ (see precise definitions in \eqref{eq:redspaces}).
	Dually, this also defines the control-reduced system $(\E\cred,\A\cred)$ defined from $\Mv\cred$ to $\Vv\cred$.
	These reductions may be iterated, and, crucially, the reductions commute (see \autoref{fig:commutativity}).
	}
\label{fig:reductions}
\end{figure}

This process of reduction is iterated, producing systems $(\E\ored[k],\A\ored[k])$ and subspaces $\Mv\ored[k]$ and $\Vv\ored[k]$.
This process ultimately stops, and we will call the number of steps before it stops the \emph{observation index}.
When the process stops, the system which is produced, denoted by $(\E\ored[\infty],\A\ored[\infty])$, is such that $\E\ored[\infty]$ is surjective.
% After running the reduction algorithm once more on the dual of that reduced system, i.e., on $(\E\oreds,\A\oreds)$, one obtains an isolated system $(\E\oredss,\A\oredss)$ such that $\E\oredss$ is now invertible.

There is a similar description of the control reduction.
When the system cannot be control-reduced, it means that the operator $\E$ is injective.
The \emph{control index} is the number of reductions needed before the system is no longer control-reducible.

Before proceeding further, we notice a fundamental property of the reductions.
As it stands now, the \emph{order} of the reductions could matter.
For instance, a control reduction followed by an observation reduction could lead to different, non-isomorphic systems.
Remarkably, such is not the case, and \emph{the reductions commute}.
Even more remarkable is that this commutation property of the reductions still holds in the much more general setting of abelian categories.
This property is summarized in \autoref{fig:commutativity}.

\tikzset{redline/.style={red, ultra thick}}

\newcommand*\drawgrid{
	\draw[step=1cm,semitransparent, gray,thin] (0,0) grid (4,4);
	% \draw[step=.5cm,gray,thin] (0,0) grid (2.5,2.5);
	\draw[thick,->,] (0,0) --  (4.5,0);
	\node[right] at (4.5,0) {obs.};
	\draw[thick,->,] (0,0) --  (0,4.5);
	\node[below, anchor=east, rotate=90] at (0,4.5) {cont.};
}

\begin{figure}
	\centering	
\begin{tikzpicture}[yscale=-1]
	\drawgrid	
	\begin{scope}[scale=2]
	\draw[ultra thick,DarkOrchid] (0,0) -- (.5,0) -- (.5,.5) -- (.5,1) -- (1,1) -- (1.5,1);
	\draw[ultra thick,DodgerBlue] (0,0) -- (0,.5) -- (.5,.5) -- (1,.5) -- (1.5,.5) -- (1.5,1);
	\fill[DimGrey] (1.5,1) circle (.8mm);
	\node[right] at (2,2) {$(\E\rred[4][4],\A\rred[4][4])$};
	\node[above] at (2,0) {$(\E\ored[4],\A\ored[4])$};
	\node[left] at (0,2) {$(\E\cred[4],\A\cred[4])$};
	\node[above right] at (1.5,1) {$(\E\rred[3][2],\A\rred[3][2])$};
\end{scope}
	
\end{tikzpicture}
\caption{
	A schematic illustration of reduction commutativity.
	The result of \autoref{prop:commutativity} is that the order of the reduction does not matter.
	Suppose that one applies three observation-reduction followed by two control-reductions, leading to $\paren[\big]{(\E\ored[3])\cred[2],(\A\ored[3])\cred[2]}$.
	Then the same system would be obtained by following either of the paths in the picture above.
	We denote the resulting system by $(\E\rred[3][2],\A\rred[3][3])$.
}
\label{fig:commutativity}
\end{figure}
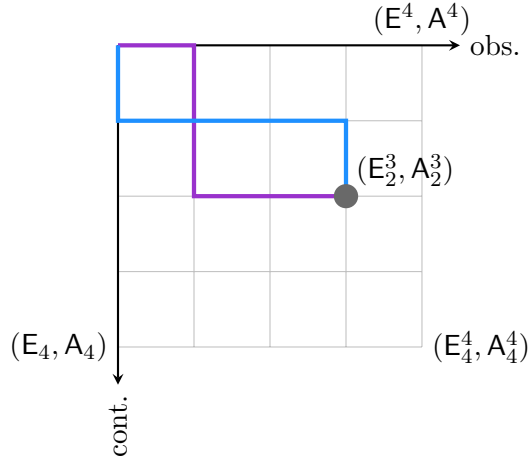

At each step of the reduction, some information from the original system is lost.
That information is encoded by integers called ``defects''.
For a given system $(\E,\A)$, those defects are of three kinds: $\nildef[0]$, $\beta\ored[0]$ and $\beta\cred[0]$.
The defect $\nildef[0]$ is defined as the dimension of the image of $\prline$, which is defined as restriction $\A$ restricted on $\ker\E$ and projected on $\coker\E$.
The defect $\beta\ored[0]$ is defined as the dimension of the cokernel of $\pline$, defined as the projection of $\A$ on $\coker\E$.
Dually, the defect $\beta\cred[0]$ is defined as the dimension of the kernel of $\rline$, defined as the restriction of $\A$ on $\ker\E$.

Both defects $\beta\ored[0]$ and $\beta\cred[0]$ have an intuitive interpretation.
As $\pline$ is the operator $\A$ followed by a projection on $\coker\E$, the space $\coker\pline$ is the space which is both in the cokernel of $\E$ and $\A$.
In matrix notation, this would correspond to rows of zeros for both matrices.
We call it the \emph{observation defect}, as the corresponding equations are observations (since they are in the cokernel of $\E$), but they observe nothing at all (as they are in the cokernel of $\A$ as well).

Dually, the control defect corresponds to the kernel of $\rline$, where $\rline$ is $\A$ restricted to $\ker\E$, so it is the space where both $\E$ and $\A$ are zero.
In matrix notation, it would correspond to columns of zeros for both matrices.
We call it the \emph{control defect} as the corresponding variables are controls (since they are in the kernel of $\E$), but then have no effect at all (as they are in the kernel of $\A$).

We know that the reductions commute, so reduced systems are characterized by two integers, recording how many observation and how many control reductions were performed.
At first glance, it would seem that we need to keep track of the defects of all those reduced systems, i.e., $\nildef\ored[0](\E\rred[m][n])$, $\beta\ored[0](\E\rred[m][n])$ and $\beta\cred[0](\E\rred[m][n],\A\rred[m][n])$.
Remarkably again, such is not the case, as we have the following simplifications, which are summarized on \autoref{fig:defects}:
\begin{itemize}
	\item $\beta\ored[0](\E\rred[m][n],\A\rred[m][n])  = \beta\ored[0](\E\ored[m],\A\ored[m]) $
	\item $\beta\cred[0](\E\rred[m][n],\A\rred[m][n])  = \beta\cred[0](\E\cred[n],\A\cred[n]) $
	\item $\nildef\ored[0](\E\rred[m][n],\A\rred[m][n])  = \nildef\ored[0](\E\ored[m+n],\A\ored[m+n]) = \nildef\ored[0](\E\cred[m+n],\A\cred[m+n]) $
\end{itemize}

\newcommand*\draworedline[1]{
		\node[rotate=90] at (#1,0) {$\coker\pline[\A\ored[#1]]$};
		\draw[redline] (#1,0) -- (#1,4);
}

\newcommand*\drawcredline[1]{
		\node at (0,#1) {$\ker\rline[\A\cred[#1]]$};
		\draw[redline] (0,#1) -- (4,#1);
}

\newcommand*\drawnilline[1]{
	\node[anchor=west, rotate=90] at (#1,0) {$\im\prline[\A\cred[#1]]$};
	\node[anchor=east] at (0,#1) {$\coim\prline[\A\ored[#1]]$};
	\draw[redline] (0,#1) -- (#1,0);
}

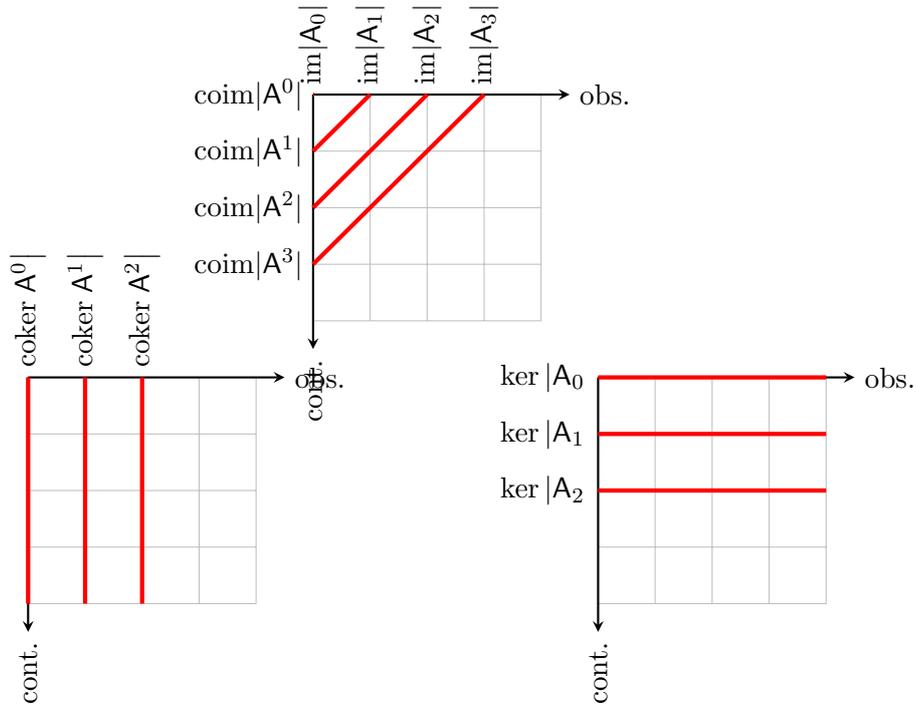
\begin{figure}
	\centering	
\begin{tikzpicture}[yscale=-1, scale=.75]
	\begin{scope}[xshift=-5cm, yshift=5cm, anchor=west]
		\drawgrid
		\draworedline{0}
		\draworedline{1}
		\draworedline{2}
	\end{scope}
	\drawgrid	
	\drawnilline{0}
	\drawnilline{1}
	\drawnilline{2}
	\drawnilline{3}
	\begin{scope}[xshift=5cm, yshift=5cm, anchor=east]
		\drawgrid
		\drawcredline{0}
		\drawcredline{1}
		\drawcredline{2}
	\end{scope}

\end{tikzpicture}
\caption{
	A description of the property of defects with respect to reduction; the proper statement is in \autoref{prop:defcomm}.
	Loosely speaking, an observation reduction leaves the sequence of \emph{control} defects unchanged, while a control reduction leaves the sequence of \emph{observation} defects unchanged.
	An observation or a control reduction ``consumes'' exactly the same nilpotency defect in the sense that the sequence of nilpotency defects is shortened by one after either reduction.
	The general idea is that those defects are a priori ``two-dimensional'' (i.e, depending on both the number of control and observation reductions), but turn out to be ``one-dimensional'' (i.e., depending only on a combination of the number of observation and control reductions), as the straight lines in the pictures above are to be interpreted as level curves.
	% The result of \autoref{prop:commutativity} is that the order of the reduction does not matter.
	% Suppose that the system is reduced after three control-reduction followed by two observation-reductions.
	% Then the same system would be obtain by following either of the paths in the picture above.
}
\label{fig:defects}
\end{figure}

One can show the following facts pertaining to the reductions:
\begin{itemize}
	\item The system is observation-irreducible if and only if $\E$ is surjective (and control-irreducible if and only if $\E$ is injective)
	\item The system is irreducible if and only if $\E$ is invertible
	\item The system is irreducible if and only if all the defective spaces vanish
	\item The order of the reduction of each type does not matter
	% \item An observation reduction leave the control spaces unchanged (and, dually, a control reduction leaves the observation spaces unchanged)
	% \item the operator $\E$ is invertible if and only if all the defects vanish
	% \item the pair $(\E,\A)$ is a regular pencil if and only if the $\beta^+$ and $\beta^-$ defects vanish
	% \item we show that the invariants defined in \cite{Mehrmann}, like the strangeness, may also be defined directly in an invariant manner, i.e., without using any extra structure or basis
\end{itemize}

% We also show that the defects and the system $(\E\redss,\A\redss)$ completely characterize the equivalence class corresponding to the equivalence relation \eqref{eq_equiv_rel}.
% \begin{itemize}
	% \item the defects are related to the Kronecker indices
	% \item the invariants defined in \cite{Mehrmann} may be used to construct a corresponding canonical form for weak equivalence: this connects the approaches of \cite{Wilkinson} and \cite{Mehrmann}
	%% \item to prove the previous claim, we construct a canonical form, of a new kind, and show that it is a rearrangement of the Kronecker decomposition
	% \item using the relation with the Kronecker decomposition theorem, we show that the defects of the dual system $(\E^*,\A^*)$ are related to those of $(\E,\A)$ by switching the $\beta^+$ and $\beta^-$ defects.
% \end{itemize}

For the purpose of the generalization, we regard those operators as \emph{morphisms} in the category of finite dimensional vector spaces instead.
We now look at what can be done to reduce a pair of morphisms from spaces $\Mv$ to $\Vv$, where $\Mv$ and $\Vv$ are now objects in the category considered: for instance, abelian groups, or $R$-modules for some ring $R$.

The reduction procedures, described above for vector spaces, turn out to make sense in such a general framework.
There are still two possible reductions, that we still call observation and control reductions.
For vector spaces,
the defects can be interpreted as dimensions of some invariant subspaces.
This is because the dimension of a space is indeed its only invariant, but such is not the case in other abelian categories.
We have therefore to replace the defects by defective spaces, and more generally, by defective objects in the category at hand.

We prove the following surprising results:
\begin{itemize}
	\item The observation and control reductions commute, as illustrated in \autoref{fig:commutativity}.
	\item One can define a sequence of defective objects (which we still call defects) which replaces the sequence of integer defects, and which behaves as in \autoref{fig:defects} with respect to the reductions.
	% \item An observation reduction leaves the sequence of control spaces unchanged, and vice-versa
	% \item The sequence of nilpotent spaces after a control reduction is the same as after an observation reduction
\item The defective spaces are related to the index in the same way: $\E$ is surjective if and only if all the observation and nilpotency defects are zero, and $\E$ is injective if and only if all the control and nilpotency defects are zero.
\item We define a system to be regular if all the observation and control defects vanish; in that case, the observation and control indices coincide.
\end{itemize}

In the category of $R$-modules, one can also define the equivalent of a resolvent set (see~\autoref{def:resolventset}):
\begin{align}
	\resolvent(\E,\A) \coloneqq \setc{\lambda \in \CC}{\lambda \E + \A \quad \text{is invertible}}
.
\end{align}
The following then holds:
\begin{itemize}
\item 
If any of the control or observation spaces are non-zero, then the resolvent is empty (see~\autoref{prop:regpencil}).
This should be reminiscent of the \emph{regular pencil theorem}~\cite[\S\,XII.2]{Gantmacher}.
\item
If, for some integers $m$ and $n$, the reduced system $(\E\rred[n][m], \A\rred[n][m])$ has a non-empty resolvent, and if all the control and observation spaces are zero, then the system $(\E,\A)$ has a non-empty resolvent (see~\autoref{prop:regpencilcor}).
\end{itemize}
We should also point out essential differences between the generalized approach and the usual decomposition in vector spaces.

\begin{itemize}
	\item The defects, which are integers in the vector space category, are replaced by objects in the category considered
	\item The reduction is not guaranteed to stall after a finite number of steps, i.e., the observation or control indices may be infinite
	\item For an irreducible system, one ends up with the characterization of the endomorphism $\E\inv A$, from $\Mv$ to $\Mv$, but this characterization may be much more involved (or even impossible) than the Jordan decomposition.
\end{itemize}

% \begin{columns}
% \begin{column}{.3\textwidth}
% \[
% \begin{split}
% x_1' &= \control{u}\\
% x_2' &= x_1\\
% &\vdots \\
% x_n' &= x_{n-1}\\
% \constraint{0 = x_n}
% \end{split}
% \]
% link control-constraint
% \end{column}
% \begin{column}{.3\textwidth}
% \[
% \begin{split}
% x_1' &= \control{u}\\
% x_2' &= x_1\\
% &\vdots \\
% x_{n+1}' &= x_n
% \end{split}
% \]
% Controlled variable
% \end{column}
% \begin{column}{.3\textwidth}
% \[
% \begin{split}
% x_1' &= 0\\
% x_2' &= x_1\\
% &\vdots \\
% \constraint{0 = x_n}
% \end{split}
% \]
% Observed variable
% \end{column}
% \end{columns}

\section{Reduction}
\label{sec_reduction}

We define here the notion of reduction of a pair of morphisms in an abelian category.

\subsection{Kernel and Images}

In an abelian category, the kernel, cokernel, image and coimage are defined for any morphism.

We now recall what those spaces are for vector spaces.
The \demph{kernel} of an operator $\E\colon\Mv\to\Vv$ is defined by
\(
	\ker\E = \setc{x \in \Mv}{\E x = 0}
\)
and its \demph{image} $\im\E$ is
\(
	\im\E = \setc{\E x}{x \in \Mv}
\).
The dual constructions are the \demph{cokernel}, defined by
\(
	\coker\E = \Vv / \im \E
\),
and the \demph{coimage}, defined by
\(
	\coim\E = \Mv / \ker \E
\).

In abelian categories, a morphism $\E\colon\Mv\to\Vv$ induces a surjective morphism 
\( \E \colon \Mv \to \im \E
	 \),
that we will still denote by $\E$.
Dually, $\E$ induces an injective morphism
\(
	\E \colon \coim\E \to \Vv
\),
that we will also still denote by $\E$.

A version of the rank-nullity Theorem in abelian categories is that the morphism $\E$ regarded as an operator from $\coim\E$ to $\im\E$ is an isomorphism.
In particular, we have $\coim\E = 0 \iff \im\E = 0$.

In the category of vector spaces, one can understand the coimage and cokernel using the adjoint.
If  $\E^* \colon \Vv^* \to \Mv^*$ denotes the adjoint operator, we have the natural identifications
\(
	\im \E^* = \coim \E  
\)
and
\(
\ker\E^* = \coker \E
\).

\subsection{System Reduction}
\label{red_space_sec}

Given an abelian category, we will call a pair of morphism $(\E,\A)$ a \alert{system}, if $\E$ and $\A$ are morphisms with the same domain and codomain:
\begin{align}
\E,\A \colon \Mv\longrightarrow \Vv
.
\end{align}

We define the operators
\begin{subequations}
	\begin{align}
		\label{eq:defpline}
		\pline \colon \Mv \to \coker \E
	\end{align}
	and
	\begin{align}
		\label{eq:defrline}
		\rline \colon \ker \E \to \Vv
	\end{align}
\end{subequations}
by projection and restriction of $\A$ respectively.

We then define the following \demph{reduced spaces} as
\begin{equation}
\label{eq:redspaces}
\begin{aligned}
	\Mv\ored  &= \ker\pline&
	\Mv\cred &= \coim \E \\
	\Vv\ored &= \EU&
	\Vv\cred &= \coker{\rline}
\end{aligned}
\end{equation}

\begin{proposition}
	In view of definitions~\eqref{eq:defpline} and~\eqref{eq:defrline}, the following diagrams commute and uniquely define the operators $\E\ored$, $\E\cred$, $\A\ored$, $\A\cred$.
See \autoref{fig:reductions} for a schematic description of those spaces and of the reduction.

\begin{subequations}
\begin{align}
	\label{eq:diagdefored}
\begin{tikzpicture}[ineq]
\matrix(m) [commdiag]%, column sep=3.5em, row sep=4em]
{0 \& \Uv\ored \& \Uv \& \coim \pline \& 0\\
0 \& \W\ored \& \W \& \coker{E} \& 0 \\};
\path[->]
(m-1-2) edge  (m-1-3)
(m-2-2) edge  (m-2-3)
(m-1-2) edge node[auto] {$\E\ored,\A\ored$} (m-2-2)
(m-1-3) edge node[auto] {$\E,\A$} (m-2-3)
(m-1-3) edge  (m-1-4)
(m-2-3) edge  (m-2-4)
(m-1-4) edge node[auto] {$0,\pline$} (m-2-4)
(m-1-1) edge (m-1-2)
(m-2-1) edge (m-2-2)
(m-1-4) edge (m-1-5)
(m-2-4) edge (m-2-5);
\end{tikzpicture}
\end{align}

\begin{align}
	\label{eq:diagdefcred}
\begin{tikzpicture}[ineq]
\matrix(m) [commdiag]%, column sep=3.5em, row sep=4em]
{0 \& \kE \& \Uv \& \Uv\cred \& 0\\
0 \& \im \rline \& \W \& \W\cred \& 0 \\};
\path[->]
(m-1-2) edge  (m-1-3)
(m-2-2) edge  (m-2-3)
(m-1-2) edge node[auto] {$0,\rline$} (m-2-2)
(m-1-3) edge node[auto] {$\E,\A$} (m-2-3)
(m-1-3) edge  (m-1-4)
(m-2-3) edge  (m-2-4)
(m-1-4) edge node[auto] {$\E\cred,\A\cred$} (m-2-4)
(m-1-1) edge (m-1-2)
(m-2-1) edge (m-2-2)
(m-1-4) edge (m-1-5)
(m-2-4) edge (m-2-5);
\end{tikzpicture}
\end{align}
% \begin{align}
% 	% \label{eq:diagdefcred}
% \begin{tikzpicture}[ineq]
% \matrix(m) [commdiag]%, column sep=3.5em, row sep=4em]
% {
% 0 \&  \W\cred \&\W \& \im \rline \& 0 \\
% 	0 \& \Uv\cred \& \Uv \& \kE \&  0\\
% };
% \path[<-]
% (m-1-2) edge  (m-1-3)
% (m-2-2) edge  (m-2-3)
% (m-1-2) edge node[auto] {$\E\cred,\A\cred$} (m-2-2)
% (m-1-3) edge node[auto] {$\E,\A$} (m-2-3)
% (m-1-3) edge  (m-1-4)
% (m-2-3) edge  (m-2-4)
% (m-1-4) edge node[auto] {$0,\rline$} (m-2-4)
% (m-1-1) edge (m-1-2)
% (m-2-1) edge (m-2-2)
% (m-1-4) edge (m-1-5)
% (m-2-4) edge (m-2-5);
% \end{tikzpicture}
% \end{align}
\end{subequations}

We will call the procedure leading to the new system
\begin{align}
	\E\ored,\A\ored \colon \Mv\ored \to \Vv\ored
\end{align}
the \demph{observation reduction}.

Similarly, the procedure leading to
\begin{align}
	\E\cred,\A\cred \colon \Mv\cred \to \Vv\cred
\end{align}
will be called the \demph{control reduction}.

\end{proposition}
\begin{proof}
	For the observation reduction, the right part of the diagram indeed commutes, simply by definition of $\pline$, and because $\E$ is zero when projected to $\coker\E$.
	The rows are exact by definition of $\Mv\ored$ and $\Vv\ored$.
	It is then a general result, obtained by diagram chasing, that the operators $\E\ored$ and $\A\ored$ are well defined, from $\Mv\ored$ to $\Vv\ored$.
	The claim for the control reduction is dual.
\end{proof}

We will denote by $(\E\ored[n],\A\ored[n])$ and $(\E\cred[n],\A\cred[n])$ the iterated observation and control reductions.
% As these operation necessarily stall, we will denote by $(\E\ored[\infty],\A\ored[\infty])$ and $(\E\cred[\infty], \A\cred[\infty])$ the corresponding operators obtained after sufficiently many reductions.

For the moment, the order of the reduction matters, but, as we shall see in \autoref{sec:commutativity}, the reduction procedures actually commute.
This will allow to make sense of the system
$(\E\ored[n]\cred[m], \A\ored[n]\cred[m])$,
which is a system observation-reduced $n$ times and control-reduced $m$ times, in whichever order.

\subsection{Generalized Kronecker Decomposition}

Before stating the decomposition result, we need a preliminary Lemma.
We first define
\begin{align}
	\prline \colon \ker \E \to \coker \E
\end{align}
by restricting $\A$ on $\ker\E$ and projecting the result to $\coker\E$.

\begin{lemma}
\label{prop:kercoker}
The following sequence is exact.
\begin{align}
\begin{tikzpicture}[ineq, ampersand replacement=\&]
\matrix(m) [commdiag,]
{
	0 \& \ker\E\ored \& \kE \& \coker\E \& \coker\E\cred \& 0\\
};
\path[->]
(m-1-1) edge (m-1-2)
(m-1-2) edge (m-1-3)
(m-1-3) edge node[auto] {$\prline$} (m-1-4)
(m-1-4) edge (m-1-5)
(m-1-5) edge (m-1-6);
\end{tikzpicture}
\end{align}
In other words, $\ker\E\ored = \ker\prline$, and $\coker\E\cred = \coker\prline$.
\end{lemma}
\begin{proof}
	We only prove the left part of the diagram, as the right part is obtained by duality.
	The proof relies on the observation that $\ker\E\ored = \ker\E \cap \Mv\ored$, and $\ker\prline = \ker\pline \cap \ker\E$.
	One concludes using that, by definition, $\Mv\ored = \ker\pline$.
\end{proof}

The following Theorem is the main ingredient of the generalized Kronecker decomposition defined in \autoref{prop:genkronecker}.

\begin{theorem}
\label{prop:kronecker}
	The following diagrams are exact.
	% We indicate in brackets the dimension of the spaces associated to the defects $\nildef[1]$ and $\beta_1$.
	We emphasize the spaces $\coker\pline$ and $\coim\prline$ which play a major role in the sequel (see \autoref{sec:defectspaces}).
	Particular emphasis is placed on the sequences in bold.
	The diagrams for the dual version of this Proposition  are in \autoref{sec:appendix}.
\begin{align}
	\label{diag:betaored}
\begin{tikzpicture}[ineq]
\matrix(m) [smallcommdiag, ]
{
	\& 0 \& 0 \& \&\\
	 0 \& \ker\A\ored  \& \ker\A  \&  0 \&\\
	0 \& \Uv\ored \& \Uv \& \coim \pline \& 0\\
0 \& \W\ored \& \W \& \coker{E} \& 0 \\
0 \& \coker\A\ored \& \coker\A \& \diagdim{\coker\pline}[\beta\ored] \& 0\\
\& 0 \& 0 \& 0 \&\\
};
\path[->, amph]
(m-2-1) edge (m-2-2)
(m-2-2) edge (m-2-3)
(m-2-3) edge (m-2-4)
;
\path[->]
(m-3-1) edge (m-3-2)
(m-3-2) edge (m-3-3)
(m-3-3) edge (m-3-4)
(m-3-4) edge (m-3-5)
;
\path[->]
(m-4-1) edge (m-4-2)
(m-4-2) edge (m-4-3)
(m-4-3) edge (m-4-4)
(m-4-4) edge (m-4-5)
;
\path[->, amph]
(m-5-1) edge (m-5-2)
(m-5-2) edge (m-5-3)
(m-5-3) edge (m-5-4)
(m-5-4) edge (m-5-5)
;
\path[->]
(m-1-2) edge (m-2-2)
(m-2-2) edge (m-3-2)
(m-3-2) edge node[auto] {$\A\ored$} (m-4-2)
(m-4-2) edge (m-5-2)
(m-5-2) edge (m-6-2)
;
\path[->]
(m-1-3) edge (m-2-3)
(m-2-3) edge (m-3-3)
(m-3-3) edge node[auto] {$\A$} (m-4-3)
(m-4-3) edge (m-5-3)
(m-5-3) edge (m-6-3)
;
\path[->, amph]
(m-2-4) edge (m-3-4)
(m-3-4) edge node[auto] {$\pline$} (m-4-4)
(m-4-4) edge (m-5-4)
(m-5-4) edge (m-6-4)
;
\end{tikzpicture}
\end{align}
\begin{align}
	\label{diag:alphaored}
	\begin{tikzpicture}[ineq]
		\matrix(m)[smallcommdiag]
		{ \& 0 \& 0 \& 0 \& \\
		0 \& \ker\E\ored \& \ker\E \& \diagdim{\coim\prline}[\nildef[1]] \& 0\\
		0 \& \Mv\ored \& \Mv \& \coim\pline \& 0 \\
		0 \& \Vv\ored[2] \& \Vv\ored \& \coker\E\ored \& 0 \\
		 \& 0 \& 0 \& 0 \&  \\};
		\path[->]
		(m-1-2) edge (m-2-2)
		(m-2-2) edge (m-3-2)
		(m-3-2) edge node[auto] {$\E\ored$} (m-4-2)
		(m-4-2) edge (m-5-2)
		;
		\path[->]
		(m-1-3) edge (m-2-3)
		(m-2-3) edge (m-3-3)
		(m-3-3) edge node[auto] {$\E$} (m-4-3)
		(m-4-3) edge (m-5-3)
		;
		\path[->, amph]
		(m-1-4) edge (m-2-4)
		(m-2-4) edge (m-3-4)
		(m-3-4) edge node[auto] {$\pline[\E]$} (m-4-4)
		(m-4-4) edge (m-5-4)
		;
		% \path[->]
		% (m-2-5) edge (m-3-5)
		% (m-3-5) edge (m-4-5)
		% (m-4-5) edge (m-5-5)
		% ;
		\path[->, amph]
		(m-2-1) edge (m-2-2)
		(m-2-2) edge (m-2-3)
		(m-2-3) edge  (m-2-4)
		(m-2-4) edge (m-2-5)
		;
		\path[->]
		(m-3-1) edge (m-3-2)
		(m-3-2) edge (m-3-3)
		(m-3-3) edge  (m-3-4)
		(m-3-4) edge (m-3-5)
		;
		\path[->]
		(m-4-1) edge (m-4-2)
		(m-4-2) edge (m-4-3)
		(m-4-3) edge (m-4-4)
		(m-4-4) edge (m-4-5)
		;
	\end{tikzpicture}
\end{align}
\end{theorem}
\begin{proof}
\begin{enumerate}
	% \item
% 	The fact that $\ker\E\ored$ is a consequence of the following diagram, where the second line is exact by mere definition of $\Mv\ored$.
% Similarly, the exactness of the remaining part of the sequence is proved using the diagram dual to the previous one:
% \begin{align}
% \begin{tikzpicture}[ineq]
% \matrix(m) [commdiag, ]
% { \& \ker\E \& \Vv \& \Vv\cred \& 0\\
%  \& \ker\E \& \coker\E \& \coker{E\cred} \& 0 \\};
% \path[->]
% (m-1-2) edge node[auto] {$\rline$} (m-1-3)
% (m-2-2) edge node[auto] {$\prline$} (m-2-3)
% (m-1-2) edge [diagequal] (m-2-2)
% (m-1-3) edge  (m-2-3)
% (m-1-3) edge  (m-1-4)
% (m-2-3) edge  (m-2-4)
% (m-1-4) edge  (m-2-4)
% % (m-1-1) edge (m-1-2)
% % (m-2-1) edge (m-2-2)
% (m-1-4) edge (m-1-5)
% (m-2-4) edge (m-2-5);
% \end{tikzpicture}
% \end{align}
% 
\item
% The other relations are a consequence of the following diagram, which are exact by definition of the defects $\beta\ored$ and $\beta\cred$.
The diagram~\eqref{diag:betaored} is the completion of the diagram~\eqref{eq:diagdefored}.
The horizontal bold arrows are uniquely defined by the completion.
The lower right horizontal arrow is automatically a surjection, and the upper left horizontal arrow is automatically an injection.
The remaining arrow is the lower left horizontal one.
One shows by diagram chasing that it is an injection, using that the upper right horizontal arrow is a surjection.
% The last row gives the identity \eqref{diag:betaored}.
% The first row gives the identity \eqref{eq:kerAequal}.
\item
In the diagram~\eqref{diag:alphaored}, the last two rows are the definitions of $\Mv\ored$ and $\Vv\ored[2]$ in~\eqref{eq:diagdefored}.
Recall that $\Vv\ored = \im\E$, so the vertical arrows are surjections.
This defines a surjective operator $\pline[\E]$ from $\coim\pline$ to $\coker\E\ored$.
The first row, with $\ker\pline[\E]$ replacing $\coim\prline$, is then the completion of that diagram by taking the kernels of the vertical maps.
Note again that the upper left horizontal arrow is automatically an injection, and the upper right horizontal arrow is a surjection by diagram chasing, using that the lower left horizontal arrow is an injection.
This gives the identification $\ker\pline[\E] = \ker\E / \ker\E\ored$.
We conclude by observing that, due to \autoref{prop:kercoker}, we have $\coim\prline = \ker\E / \ker\prline = \ker\E/\ker\E\ored$.
% $\coim\prline = \ker\E/\ker\E\ored$.
% It is a consequence of the following commuting diagram, as, by definition, $\Mv\ored = \ker\pline$:
% \begin{align}
% \begin{tikzpicture}[ineq]
% \matrix(m) [commdiag, ]
% {0 \& \ker\E\ored \& \ker\E \& \coker \E \\
% 0 \& \Mv\ored \& \Mv \& \coker{E}  \\};
% \path[->]
% (m-1-2) edge  (m-1-3)
% (m-2-2) edge  (m-2-3)
% (m-1-2) edge  (m-2-2)
% (m-1-3) edge  (m-2-3)
% (m-1-3) edge node[auto] {$\prline$} (m-1-4)
% (m-2-3) edge node[auto] {$\pline$} (m-2-4)
% (m-1-4) edge[diagequal]  (m-2-4)
% (m-1-1) edge (m-1-2)
% (m-2-1) edge (m-2-2)
% % (m-1-4) edge (m-1-5)
% ;
% \end{tikzpicture}
% \end{align}
\end{enumerate}
\end{proof}

The following result is essentially the counterpart of the Kronecker decomposition theorem in vector spaces.
We indicate the precise relation in \autoref{sec:veckronecker}
\begin{proposition}
\label{prop:genkronecker}
	The definition~\eqref{eq:diagdefored} induces a decomposition of $\Vv$ into $\Vv\ored$ and $\coker \E$.
	The space $\Vv\ored$ is then in turn decomposed into $\Vv\ored[2]$ and $\coker\E\ored$ by the same device.
	The space $\coker \E$ is decomposed as
\begin{align}
	\label{diag:decompcokerE}
\begin{tikzpicture}[ineq, ampersand replacement=\&]
\matrix(m) [commdiag,]
{
	0 \& \im\pline \& \coker\E \& \diagdim{\coker\pline}[\beta\ored] \& 0\\
};
\path[->]
(m-1-1) edge (m-1-2)
(m-1-2) edge (m-1-3)
(m-1-3) edge  (m-1-4)
(m-1-4) edge (m-1-5)
;
\end{tikzpicture}
\end{align}

On the other hand, the definition~\eqref{eq:diagdefored} defines a decomposition of $\Mv$ into $\Mv\ored$ and $\coim\pline$.
The space $\coim\pline$ is further decomposed into
\begin{align}
	\label{diag:decompcoimpline}
\begin{tikzpicture}[ineq, ampersand replacement=\&]
\matrix(m) [commdiag,]
{
	0 \& \diagdim{\coim\prline}[\nildef\ored] \& \coim\pline \& \coim\pline[\E] \& 0\\
};
\path[->]
(m-1-1) edge (m-1-2)
(m-1-2) edge (m-1-3)
(m-1-3) edge  (m-1-4)
(m-1-4) edge (m-1-5)
;
\end{tikzpicture}
\end{align}

Moreover, $\pline[\E]$ is a bijection from $\coim\pline[\E]$ to $\coker\E\ored$, and $\pline$ is a bijection from $\coim\pline$ to $\im\pline$.
Note that $\pline$ is zero after projection on $\coker\pline$, and $\pline[\E]$ is zero when restricted to $\coim\prline$.

\end{proposition}
\begin{proof}
	The exactness of~\eqref{diag:decompcokerE} follows the last column of~\eqref{diag:betaored}, and that $\pline$ is an isomorphism between $\coim\pline$ and $\im\pline$.
	The exactness of~\eqref{diag:decompcoimpline} follows from the last column of~\eqref{diag:alphaored}, and that $\pline[\E]$ is an isomorphism between $\coim\pline[\E]$ and $\im\pline[\E]$, which, by~\eqref{diag:decompcoimpline} is $\coker\E\ored$.
\end{proof}

Of course, \autoref{prop:genkronecker} also has a dual version.

\section{The Reductions Commute}
\label{sec:commutativity}

% As observed in \cite{linpdae}, the two reduction procedures commute.
We now establish that the two reduction procedures commute.

% \begin{figure}
	% \centering
\begin{proposition}
\label{prop:commutativity}
	The observation- and control-reduction procedures commute, in the sense that the spaces $(\Mv\ored)\cred$ and $(\Vv\ored)\cred$ are canonically isomorphic to the spaces $(\Mv\cred)\ored$ and $(\Vv\cred)\ored$.
	We denote the first space as $\Mv\ored\cred$ and the second $\Vv\ored\cred$.

	Moreover, the following two diagrams commute.
	The operator $\A$ sends the first diagram to the second one (in the sense that it preserves the commutations).
	\begin{subequations}
		\begin{align}
\ninediag{\ker\E\ored \& \ker\E \& \diagdim{\coim\prline}[\nildef]}%
{\Mv\ored \& \Mv \& \coim\pline}%
{\Mv\cred\ored \& \Mv\cred \& \coim\pline[\A\cred]}
\end{align}
\begin{align}
\ninediag{\im\rline[\A\ored] \& \im\rline \& \diagdim{\im\prline}[\nildef]}%
{\Vv\ored \& \Vv \& \coker\E}%
{\Vv\ored\cred \& \Vv\cred \& \coker\E\cred}
\end{align}
\end{subequations}
\end{proposition}
\begin{proof}
	We have already established the exactness of the first row of the first diagram.
	% $0\to\ker\E\ored\to\ker\E\to\coim\prline\to 0$.
	The operators $\rline[\A\ored]$, $\rline[\A]$ and $\prline$ send this sequence to the first row of the second diagram, which establishes its exactness.
	% $0\to\im\A\ored\to\im\rline\to\im\prline\to 0$.
	We obtain by duality the exactness of the last column of the first diagram.

	Consider now the first two rows of the first diagram.
	The completion gives the last row as $\star \to \Mv\ored[1]\cred[\ 1] \to \Mv\cred \to \coim\pline[\A\cred] \to 0$.
	As usual, we use that the upper right horizontal arrow is a surjection to establish that $\star = 0$.
	Now, the last row defines the space $\Mv\cred[1]\ored[\ 1]$.
\end{proof}

\section{Control, Observation and Nilpotency Spaces}
\label{sec:defectspaces}

We also have the following fundamental result about the defective spaces $\ker\rline$ and $\coker\pline$.
This ensures that the observation reduction $(\E,\A)\to (\E\ored,\A\ored)$ has no effect on the sequence of spaces $\ker\rline[\A\ored[n]]$.
\begin{lemma}
\label{prop:defred}
	The spaces $\ker\rline$ and $\ker\rline[\A\ored]$ are isomorphic.
\end{lemma}
\begin{proof}
	From the first line of~\eqref{diag:betacred} we have
\begin{align}
\begin{tikzpicture}[ineq, ampersand replacement=\&]
\matrix(m) [commdiag,]
{
	0 \& \ker\rline \& \ker\A \& \ker\A\cred \& 0\\
};
\path[->]
 (m-1-1) edge (m-1-2)
 (m-1-2) edge (m-1-3)
 (m-1-3) edge (m-1-4)
 (m-1-4) edge (m-1-5)
;
\end{tikzpicture}
\end{align}
By simply considering the reduced system $(\E\ored,\A\ored)$ we also obtain
\begin{align}
\begin{tikzpicture}[ineq, ampersand replacement=\&]
\matrix(m) [commdiag,]
{
	0 \& \ker\rline[\A\ored] \& \ker\A\ored \& \ker\paren{\A\ored}\cred \& 0\\
};
\path[->]
 (m-1-1) edge (m-1-2)
 (m-1-2) edge (m-1-3)
 (m-1-3) edge (m-1-4)
 (m-1-4) edge (m-1-5)
;
\end{tikzpicture}
\end{align}
We may now apply \autoref{prop:commutativity} to obtain that $\paren{\A\ored}\cred \equiv \paren{\A\cred}\ored$.
We now read the first row of~\eqref{diag:betaored} to obtain that $\ker\A\ored = \ker \A$ and that $\ker\paren{\A\cred}\ored = \ker\A\cred$.
\end{proof}

We have the following simplification:
\begin{proposition}
	The spaces $\coim \prline[\A\cred]$ and $\im\prline[\A\ored]$ are isomorphic.
\end{proposition}
\begin{proof}
From the first row of~\eqref{diag:alphaored}, and using \autoref{prop:commutativity}, we obtain $\coim\prline[\A\cred]\equiv \ker\E\cred/\ker\E\cred\ored$.
Similarly from the last row of~\eqref{diag:alphacred} and using \autoref{prop:commutativity} we obtain $\im\prline[\A\ored] \equiv \coker\E\cred/\coker\E\cred\ored$.
The result is then a consequence of \autoref{prop:seqE}
\end{proof}

\begin{lemma}
\label{prop:seqE}
	The morphism $\E$ induces the following exact sequence.
\begin{align}
\begin{tikzpicture}[ineq, ampersand replacement=\&]
\matrix(m) [commdiag,]
{
	0 \& \ker\E\ored\cred \& \ker\E\cred \& \coker\E\ored \& \coker\E\ored\cred \& 0\\
};
\path[->]
 (m-1-1) edge (m-1-2)
 (m-1-2) edge (m-1-3)
 (m-1-3) edge (m-1-4)
 (m-1-4) edge (m-1-5)
 (m-1-5) edge (m-1-6)
;
\end{tikzpicture}
\end{align}
\end{lemma}
\begin{proof}
	We give a proof in the category of abelian groups, which automatically translates to any abelian category~\cite{Fr64}.
	We have to show that $\ker\E\ored\cred$ is the kernel of the operator induced by $\E$ in the middle, which we denote by $[\E]$.
	First, notice that
	\begin{align}
		\ker\E\ored\cred = \ker\E\cred\cap \Mv\cred\ored = \ker\E\cred \cap \Mv\ored
		.
	\end{align}
	Next, $\ker[\E] = \setc{x\in\ker\E\cred}{\E x \in \im\E\ored + \ker\E}$.
	We thus obtain $\ker\E\ored\cred \subset \ker[\E]$.
	Suppose on the other hand that $x + \ker\E\in \ker[\E]$.
	It means that there is a $y\in\Mv\ored$ such that $\E x = \E y$, from which we obtain that $\E(x + \ker\E) \in \im\E\ored$.
	This concludes the proof.
	% On the one hand, that kernel is described by
	% \begin{align}
	% 	\ker[\E] = \setc{x\in\\ker\E\cred}{\E x \in \im\E\ored} = \setc{x + \ker\E}{\E x \in \im\rline \quad \E x \in \Vv\ored[2]}
	% \end{align}
	% On the other hand, $\ker\E\cred\ored$ is defined by
	% \begin{align}
	% 	\ker\E\ored\cred = \ker\E\cred \cap \Mv\cred\ored = \ker\E\cred \cap \Mv\ored = 
	% \end{align}
\end{proof}

% If we combine \autoref{prop:commutativity} and \autoref{prop:defred}, we obtain the following remarkable simplification:
We gather our findings on the defective spaces.

\begin{proposition}
\label{prop:defcomm}
	For any integers $n$, $m$, we have (see \autoref{fig:defects}):
	
\begin{align}
	\ker\rline[\A\cred[m]\ored[n]] \equiv \ker\rline[\A\cred[m]]
	\qquad 
	\coker\pline[\A\cred[m]\ored[n]] \equiv \coker\pline[\A\ored[n]]
	\\
	\coim\prline[\A\ored[n]\cred[m]] = \coim\prline[\A\ored[n+m]] = \im\prline[A\cred[n+m]] = \im\prline[\A\ored[n]\cred[m]]
\end{align}
\end{proposition}

\section{Observation and Control Indices}

The reduction procedures produce sequences of systems.
% These sequences must stall.
As we showed in \autoref{sec:commutativity}, the order in which reductions of either type are performed is irrelevant.
We therefore define the \demph{observation index} and \demph{control index} to be
\begin{align}
	\ind\ored[*] \coloneqq \min \setc{n\in\NN}{\sys\ored[n] = \sys\ored[n+1]}
\end{align}
and
\begin{align}
	\ind\cred[*] \coloneqq \min \setc{n\in\NN}{\sys\cred[n] = \sys\cred[n+1]}
\end{align}
respectively.

\subsection{Indices and Reduction}

% When both sequences stall, the system is totally reduced.
% The number of reduction steps needed to transform a system into a totally reduced one is called the \emph{index} of the system $(\E,\A)$:
% 
% \begin{definition}
% \label{def_index}
% The smallest integer $n\in\NN$  for which the system $(\E\red{n},\A\red{n})$ is totally reduced
% is called the \alert{index} of the system $(\E,\A)$. 
% 
% We will use the following notation for the index of the system $(\E,\A)$:
% \[\ind\EA := \min \setc{ n\in\NN}{(\E,\A)\red{n+1} = (\E,\A)\red{n} }
% .
% \]
% \end{definition}

% \begin{remark}
The observation-reduced system $(\E\ored,\A\ored)$ has an index dropped by one, i.e,
\begin{align}
	\ind\ored[*]\paren{\E\ored,\A\ored} = \ind\ored[*](\E,\A) - 1
.
\end{align}
There is a dual statement for the control index and control-reduction:
\begin{align}
	\ind\cred[*]\paren{\E\cred,\A\cred} = \ind\cred[*](\E,\A) - 1
.
\end{align}

In the vector space category, both indices are \emph{always} finite integers\footnote{as opposed to the differentiation index, which is infinite in the non-regular pencil case; see, e.g.,~\cite[\S VII.1]{HaWa96}.}.

% Those observations will be used repeatedly to prove statements by induction on the index (e.g., in \autoref{prop_reg_pencil}, \autoref{thm_linear_decomposition} and \autoref{thm_NL_basis}).
% \end{remark}

% \begin{remark}
% \label{rem_def_index}
% Using \autoref{prop_tot_red_V} we observe that
% \[\ind\EA = \min \setc{ n\in\NN}{\Vv\red{n+1} = \Vv\red{n} }
% .
% \]
% \end{remark}

% \begin{remark}
% The index defined in \autoref{def_index} is closely related to the geometric index defined in \cite{Reich}, \cite{Rabier-Rheinboldt} or \cite[\S5.1]{thesis}.
% In fact, the geometric index would be the first integer $n$ such that the system $(\E,\A)\red{n}$ is \emph{almost} reduced (\autoref{def_almost_red}).
% As we shall see in \autoref{prop_ind_no_beta} and \autoref{prop_reg_pencil}, this minor difference is only relevant for singular pencils.
% \end{remark}

In order to understand the observation index, one has to answer the following question: when is a system observation-irreducible?
The following result gives a number of equivalent conditions:
\begin{proposition}
\label{prop:inddefects}
	The following are equivalent
	\begin{thmenumerate}
	\item
\label{it:indzero}
		$\ind\ored[*](\E,\A) = 0$
	\item
\label{it:cokerEzero}
		$\coker\E = 0$
	\item
\label{it:obsdefectzero}
		$\coim\prline = 0$, $\coker\pline = 0$, and $\coker\E\ored = 0$
	\item
\label{it:allobsdefectszero}
		$\ind\ored[*](\E,\A) <\infty$, $\coim\prline[\A\ored[n]] = 0$ and $\coker\pline[\A\ored[n]] = 0$ for all $n\in\NN$
\end{thmenumerate}
\end{proposition}
\begin{proof}
	By definition of $\E\ored$ and $\A\ored$, we know that $\ind\ored[*](\E,\A) = 0$ is equivalent to $\Mv\ored = \Mv$ and $\Vv\ored = \Vv$.
	Reading~\eqref{eq:diagdefored}, we see that this is equivalent to $\coker\E = 0$ and $\coim\pline = 0$.
	The last column of the diagram~\eqref{diag:betaored},  gives us
	\begin{align}
		\label{eq:prooflastcol}
	\coker\E = 0 \iff \paren[\Big]{\coim\pline = 0 \text{ and } \coker\pline = 0}
	% \coker\E = 0 \iff \pline = 0
\end{align}
	In particular, we have $\coker\E = 0 \implies \coim\pline = 0$, so we have thus obtained \autoref{it:indzero} $\iff$ \autoref{it:cokerEzero}.
	By the last column of~\eqref{diag:alphaored}, we obtain $\coim\pline = 0 \iff \coim\prline = 0 \text{ and } \coker\E\ored = 0$.
	By using~\eqref{eq:prooflastcol} again, this gives the equivalence with \autoref{it:obsdefectzero}.
	Assuming \autoref{it:indzero}, as $\A\ored[n] = \A$, and by the previous equivalences, we obtain the implication \autoref{it:obsdefectzero} $\implies$ \autoref{it:allobsdefectszero}.
	On the other hand, assume \autoref{it:allobsdefectszero}, so we have $\ind\ored[*](\E,\A) = N$.
	By successively applying \autoref{it:obsdefectzero} to the reduced systems, we obtain $\coker\E = \coker\E\ored[N] = 0$.
\end{proof}

We also have the corresponding dual result
\begin{proposition}
\label{prop:inddefectsdual}
	The following are equivalent
	\begin{thmenumerate}
	\item
		$\ind\cred[*](\E,\A) = 0$
	\item
		$\ker\E = 0$
	\item
		$\im\prline = 0$, $\ker\rline = 0$, and $\ker\E\cred = 0$
	\item
		$\ind\cred[*](\E,\A) <\infty$, $\im\prline[\A\cred[n]] = 0$ and $\ker\rline[\A\cred[n]] = 0$ for all $n\in\NN$
\end{thmenumerate}
\end{proposition}

% We have the following relation between reductivity and defects.
% \begin{proposition}
% 	\label{prop:indequiv}
% 	The following are equivalent
% 	\begin{thmenumerate}
% 	\item $\coker\E = 0$
% 	\item $\nildef[n] = \beta\ored[n] = 0\qquad n\in\NN$
% 	\item $(\E\ored,\A\ored) = (\E,\A)$
% 	\item $\ind\ored[*](\E,\A) = 0$
% 	\end{thmenumerate}
% 	And the dual version is
% 	\begin{thmenumerate}
% 	\item $\ker\E = 0$
% 	\item $\nildef[n] = \beta\cred[n] = 0\qquad n\in\NN$
% 	\item $(\E\cred,\A\cred) = (\E,\A)$
% 	\item $\ind\cred[*](\E,\A) = 0$
% 	\end{thmenumerate}
% \end{proposition}
% 
% The index is, as expected, a non-dynamical invariant.
% More precisely, it is a function of the defects, as the following proposition shows:
% \begin{proposition}
% The {index} $\ind\EA$ (see \autoref{def_index}) of a linear system $(\E,\A)$ is given by
% \[\ind\EA = \min \setc{n\in\NN}{\forall k > n\quad \nildef[k](\E,\A) =0\quad\text{and}\quad \beta^+_{k}(\E,\A) =0}
% .
% \]
% \end{proposition}
% \begin{proof}
% Following \autoref{rem_def_index}, the index fulfills
% \[\ind\EA = \min_k \dim\DV\red{k+1}=0
% .
% \]
% Using \autoref{lma_rel_defects_dim_DVDM} we thus obtain
% \[\dim \DV\red{k} = 0 \iff \nildef[j]+\beta^+_{j} =0\quad\forall j \geq  k+1
% ,
% \]
% which proves the claim.
% \end{proof}

In view of \autoref{prop:inddefects} and \autoref{prop:inddefectsdual}, we have
\begin{corollary}
\label{prop:indmax}
	The observation and control indices are related to the defective spaces by
	\begin{align}
		\ind\ored[*] &= \max\setc{n\in\NN}{\coim\prline[\A\ored[n]] \neq 0 \quad\text{or}\quad \coker\pline[\A\ored[n]] \neq 0}\\
		\ind\cred[*] &= \max\setc{n\in\NN}{\im\prline[\A\cred[n]] \neq 0 \quad\text{or}\quad \ker\rline[\A\cred[n]] \neq 0}
	\end{align}
\end{corollary}

In the case of a regular pencil, we obtain
\begin{proposition}
If $(\E,\A)$ is a regular pencil, i.e., $\ker\rline[\A\cred[n]] = 0$ and $\coker\pline[\A\ored[n]] = 0$ for any integer $n$,  then the observation and control indices coincide:
\begin{align}
	\ind\cred[*](\E,\A) = \ind\ored[*](\E,\A)
\end{align}
\end{proposition}
\begin{proof}
	From \autoref{prop:defcomm} we obtain that $\coim\prline[\A\ored[n]] = \coim\prline[\A\cred[n]] = \im\prline[\A\cred[n]]$.
	We conclude using \autoref{prop:indmax}.
\end{proof}

Note that there are systems for which the control and observation indices coincide, but which are not regular pencils.

\subsection{Iterated Reductions}

% The reduction procedure may thus be used once to obtain a totally reduced system, and may then be applied again to the dual of that totally reduced system.

% Starting with a system $(\E,\A)$, we may completely reduce it to obtain the system $(\E\red{\infty},\A\red{\infty})$.
% The operator $\E\reds$ is injective.
% The adjoint system $(\E\reds,\A\reds)$ may be in turn completely reduced to obtain the system $(\E\redss,\A\redss)$.
% Using \autoref{prop_Eredinf_surj} and \autoref{prop_Ered_inj}, we obtain the following result.

% Repeated application of both reduction procedures produce a system where $\ker\E = \coker\E = 0$. 

In general, there is no guarantee than either of the reduction procedures will stall.
We record some evidence that these reduction simplify the system at hand.

\begin{proposition}
\label{prop_Edyn_inv}
If both reduction stall after a finite number of steps, the operator $\E\rred[\infty][\infty]$ is invertible.
% , in the sense that
% \begin{align}
% 	\ker\E\redsss = 0 \qquad \coker\E\redsss = 0
% \end{align}
\end{proposition}

% \begin{definition}
% \label{def_dyndim}
% The \alert{dynamical dimension} $\delta$ of the system $(\E,\A)$ is defined by the integer
% \[\delta := \dim\Mv\redsss = \dim\Vv\redsss
% .
% \]
% \end{definition}

\begin{remark}
In the category of finite dimensional vector spaces, since the operator $\E\rred[\infty][\infty]$ is invertible, its domain and co-domain have the same dimension. 
This dimension is the dimension of the intrinsic dynamics of the system.

For a differential equation defined by the system $(\E,\A)$, the system 
\[(\E\rred[\infty][\infty],\A\rred[\infty][\infty])\]
 corresponds to the underlying differential equation.
In particular, the  dimension
\begin{align}
	\delta \coloneqq \dim\Mv\rred[\infty][\infty] = \dim\Vv\rred[\infty][\infty]
\end{align}
 determines the degrees of freedom for the choice of the initial condition.
\end{remark}

Introduce for simplicity the notation
\begin{align}
	\DM\ored[k] \coloneqq \Mv\ored[k-1]/\Mv\ored[k] 
	\qquad\text{and}\qquad
	\DV\ored[k] \coloneqq \Vv\ored[k-1]/\Vv\ored[k]
	.
\end{align}
Reading the diagrams of \autoref{prop:kronecker}, we see that $\pline$ is an injection from $\DM\ored[k]$ to $\DV\ored[k]$, and that $\pline[\E]$ is a surjection from $\DM\ored[k]$ to $\DV\ored[k+1]$, for any integer $k\geq 1$.
This implies the following \emph{injections}, which conveys the idea that the effect of each reduction becomes smaller and smaller.
\begin{align}
\begin{tikzpicture}[ineq, ampersand replacement=\&]
\matrix(m) [smallcommdiag,]
{
	\cdots \& \DM\ored[k+1] \&  \dim \DV\ored[k+1] \&   \DM\ored[k] \&   \DV\ored[k]\& \cdots \\
};
\path[right hook->]
(m-1-1) edge (m-1-2)
(m-1-2) edge (m-1-3)
(m-1-3) edge  (m-1-4)
(m-1-4) edge (m-1-5)
(m-1-5) edge (m-1-6)
;
\end{tikzpicture}
\end{align}

This effect is of course only made precise in the category of finite-dimensional vector spaces, where this implies the inequalities
% \begin{align}
% 	\cdots \leq \dim\coim\pline\ored[k+1] \leq \dim \coker\E\ored[k+1] \leq \dim \coim\pline\ored[k] \leq \dim\coker\E\ored[k]\leq \cdots
% \end{align}
% which is equivalent to
\begin{align}
	\cdots \leq \dim\DM\ored[k+1] \leq  \dim \DV\ored[k+1] \leq  \dim \DM\ored[k] \leq  \dim \DV\ored[k]\leq \cdots
,
\end{align}

This is the same sequence of inequalities as in~\cite[\S\,5.2]{Wilkinson}.

\section{Defect spaces and resolvent set}
\label{sec_regular_pencil}

In this section, we assume that the category at hand is a category of $R$-modules, for a given ring $R$.
Note that any abelian category can be embedded in such a category by the Mitchell--Freyd theorem~\cite[\S\,7]{Fr64}.

\begin{definition}
\label{def:resolventset}
We define the \demph{resolvent set} of the system $(\E,\A)$ as
\begin{align}
	\resolvent(\E,\A) = \setc{\lambda \in R}{\lambda \E + \A \text{ is invertible}}
  .
\end{align}
\end{definition}

\begin{lemma}
\label{lma_regular_pencil}
% The system $(\E,\A)$ is a regular pencil if and only if both the following properties hold:
For any $\lambda \in R$, if two of the following properties hold, then so does the third.
\begin{thmenumerate}
	\item  $\lambda \E + \A$ is invertible
	\item  ${\lambda \E\ored + \A\ored}$ is invertible
	\item $\coker\pline = 0$
\end{thmenumerate}
Note that the last property is independent of $\lambda$.
% so we obtain in particular that if $\coker\pline = 0$, then $\lambda\E + \A$ is never invertible.
\end{lemma}

\begin{proof}
We define the operator $\Sop[\lambda]$  by
\[\Sop[\lambda] := \lambda\E + \A
.
\]
% $\Sop[\lambda]$ can be decomposed into $\Sop[\lambda]'$ and $[\Sop[\lambda]]$ according to the following commuting diagram:
Consider the following diagram:
\begin{align}
	\label{diag:sop}
\begin{tikzpicture}[ineq]
\matrix(m) [commdiag]
{
	\& \& \& 0 \&\\
	0 \& \Mv\ored \& \Mv \& \coim \pline \& 0 \\
0 \& \Vv\ored \& \Vv \& \coker\E \& 0 \\
\& \& \& \diagdim{\coker\pline}[\beta\ored] \& \\
\& \& \& 0 \&\\
};
\path[->]
(m-2-1) edge (m-2-2)
(m-2-2) edge (m-2-3)
(m-2-3) edge (m-2-4)
(m-2-4) edge (m-2-5)
(m-3-1) edge (m-3-2)
(m-3-2) edge (m-3-3)
(m-3-3) edge (m-3-4)
(m-3-4) edge (m-3-5)
(m-2-2) edge node[auto] {$\Sop[\lambda]\ored$} (m-3-2)
(m-2-3) edge node[auto] {$\Sop[\lambda]$} (m-3-3)
(m-2-4) edge node[auto] {$\pline$} (m-3-4);
\path[->]
(m-1-4) edge (m-2-4)
(m-2-4) edge (m-3-4)
(m-3-4) edge (m-4-4)
(m-4-4) edge (m-5-4)
;
\end{tikzpicture}
\end{align}
The diagram commutes because $\E$ projected to $\coker\E$ is zero.

From the short five lemma (or by simple diagram chasing), we obtain that out of the three operators $\Sop[\lambda]\ored$, $\Sop[\lambda]$ and $\pline$ (restricted on $\coim\pline$), if two of them are invertible then the third one is. 
% One easy way to prove this fact\footnote{This is a very general result that holds in other contexts as well, since one may also prove it by diagram chasing.} is by choosing bases in $\Mv$ and $\Vv$ which are compatible with the subspaces $\Mv'$ and $\Vv'$. The operator $\Sop[\lambda]$ is then represented by a block triangular matrix where the diagonal blocks are the matrices of $\Sop[\lambda]'$ and $[\Sop[\lambda]]$.
% Now it is easy to check that if two of those three matrices are invertible, the third one is.

% \item
% Notice that for any $\lambda\in\CC$, $[\Sop[\lambda]] = [\A]$ (the operator $[\A]$ is defined in \autoref{prop_quotient_op}), so $[\Sop[\lambda]]$ is invertible if and only if $\beta_1^+ = 0$. 
% As a result, we obtain the property
Finally,  $\pline$ restricted to $\coim\pline$ is invertible if and only if $\coker\pline = 0$.
% , that is, if and only if $\beta\ored = 0$.
% We obtain
% \[\beta\ored =0 \implies \bracket[\big]{  \forall \lambda\in\CC\quad \Sop[\lambda] \text{ invertible} \iff \Sop[\lambda]' \text{ invertible} }
% .
% \]
% \item
% For any $\lambda\in\CC$, the surjectivity of $\Sop[\lambda]$ implies that of $[\Sop[\lambda]]$. Since $[\Sop[\lambda]]$ does not depend on $\lambda$, it means that if $[\Sop[\lambda]]=[\A]$ is not surjective, then $\Sop[\lambda]$ is not surjective for any $\lambda\in\CC$. Now since, by definition, if $\beta^+_1\neq 0$ then $[A]$ is not surjective, we conclude that
% \[\beta\ored \neq  0 \implies \forall \lambda\in\CC\quad \Sop[\lambda] \text{ not invertible}
% .
% \]
% \end{enumerate}
\end{proof}

There is of course a dual statement to \autoref{lma_regular_pencil}.

We obtain the following alternative:
\begin{proposition}
\label{prop:regpencil}
	% If $\coker\pline = 0$, then  for any $\lambda \in R$, $\lambda\E + \A$ is invertible if and only if $\lambda\E\ored + \A\ored$ is, in other words, they have the same spectrum.
	% If $\coker\pline \neq 0$, then for any $\lambda \in R$, $\lambda \E + \A $ is not invertible.
	We have the following mutually exclusive alternative
	\begin{thmenumerate}
		\item if $\coker\pline = 0$ then $\resolvent(\E,\A) = \resolvent(\E\ored, \A\ored)$.
		\item if $\coker\pline \neq 0$ then $\resolvent(\E,\A) = \emptyset$.
	\end{thmenumerate}
\end{proposition}
\begin{proof}
\autoref{lma_regular_pencil} gives $\coker\pline = 0 \implies \resolvent(\E,\A) = \resolvent(\E\ored,\A\ored)$ on the one hand, and $\coker\pline = 0 \iff \resolvent(\E\ored, \A\ored) \cap \resolvent(\E,\A) = \emptyset$ on the other hand.
We see from~\eqref{diag:sop} that $\ker\Sop[\lambda]\ored \subset \ker\Sop[\lambda]$, so $\resolvent(\E,\A) \subset \resolvent(\E\ored,\A\ored)$, which proves the result.
\end{proof}

This gives the ``regular pencil'' theorem.

\begin{proposition}
\label{prop:regpencilcor}
	% Suppose that $\ind\ored[*](\E,\A) < \infty$ and $\ind\cred[*](\E,\A)<\infty$.
	Suppose that for some integer $M$ and $N$, $\resolvent(\E\ored[N]\cred[M],\A\ored[N]\cred[M]) \neq \emptyset$.
	Then $\resolvent(\E,\A) \neq \emptyset$ if and only if $\coker\pline[\A\ored[n]] = 0$ and $\ker\rline[\A\cred[m]]$ for all $n \leq N$ and $m\leq M$.
\end{proposition}

\section{Applications in the category of vector spaces}

% \subsection{Coupling Lemma}
\subsection{Defects}

% Notice further that $\rline$ is surjective, and $\pline$ is injective.
Recall the definitions of $\pline$ and $\rline$ in~\eqref{eq:defpline} and~\eqref{eq:defrline}.

We define the  \demph{observation defect} $\beta\ored$ and \demph{control defect} $\beta\cred$ by
\begin{align}
	\beta\ored &\coloneqq \dim \coker\pline & \beta\cred &\coloneqq  \dim \ker\rline
\end{align}

We define the \demph{nipotency defect} $\nildef[1]$ by
\begin{align}
	\nildef[1] \coloneqq \dim \im\prline = \dim \coim\prline
	.
\end{align}

\begin{proposition}
\label{prop_rel_defect_dims}
This in turn gives the following relations between the defects and the kernels and cokernels of the observation-reduced and control-reduced system
	\begin{align}
		\label{eq:alphadimkE}
		\nildef[1] &= \dim\ker\E - \dim \ker \E\ored \\
		&= \dim\coker\E - \dim \coker\E\cred\\
		\beta\ored &= \dim\coker\A - \dim\coker\A\ored \\
		\beta\cred &= \dim\ker\A - \dim\ker\A\cred 
	\end{align}
	Moreover, we have
	\begin{align}
		\label{eq:kerAequal}
		0 &= \dim \ker \A - \dim \ker\A\ored \\
		&= \dim \coker \A - \dim\coker\A\cred
	\end{align}
Finally, we have the dual relations
\begin{align}
	\dim \coker\E - \dim\coker\E\ored &= \nildef[1] + \beta\ored\\
	\dim \ker\E - \dim\ker\E\cred &= \nildef[1] + \beta\cred
\end{align}
	
\end{proposition}

\subsection{Relation with the Kronecker Decomposition}
\label{sec:veckronecker}

Recall that, in finite dimensional vector spaces, short exact sequences split, in other word, every subspace has a complementary subspace.
By applying this to \autoref{prop:genkronecker}, we obtain the following result:
\begin{lemma}
\label{lma_coupling_system}
Suppose that we have a decomposition $\Vv\ored = \Vv\ored[2] \oplus \sect{\coker\E\ored}$.
There are decompositions $\Mv = \Mv\ored \oplus \sect{\coim\pline[\E]} \oplus \sect{\coim\prline}$ and $\Vv = \Vv\ored \oplus \sect{\im\pline} \oplus \sect{\coker\pline}$ such that $\E$ is a bijection from $\sect{\coim\pline[\E]}$ to $\sect{\coker\E\ored}$ and $\A$ is a bijection from $\sect{\coim\pline}$ to $\sect{\im\pline}$.
\end{lemma}

The {Kronecker canonical form} makes use of special blocks, each of which having a variant for the matrices $\E$ and $\A$.

\begin{theorem}
\label{thm_Kronecker}
A decomposition with defects $\nildef[k]$, $\beta\ored[k]$, $\beta\cred[k]$, produces a Kronecker decomposition which for all integer $k\geq 1$ contains
\begin{itemize}
	\item $\nildef[k]$ block of type $\mat{N}_{k}$,
	\item $\beta\ored[k]$ blocks of type $\mat{L}_k$,
	\item $\beta\cred[k]$ blocks of type $\mat{L}_k^{\transpose}$.
\end{itemize}
\end{theorem}

\matrixfigure{fig_kronecker}{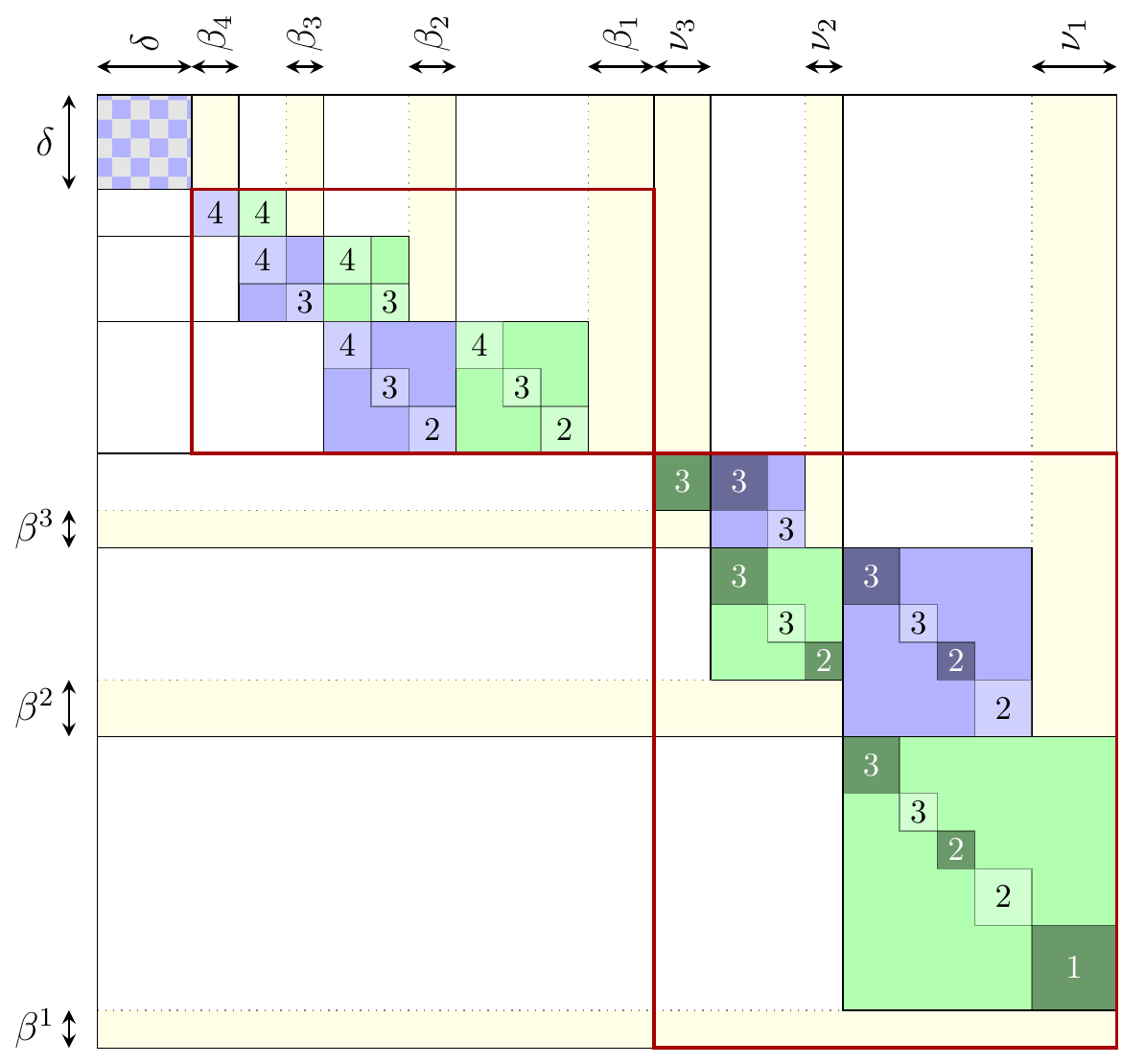}{
An illustration of the defects $\beta\ored[j]$, $\beta\cred[j]$ and $\nildef[j]$, and their relation to the Kronecker decomposition described in \autoref{thm_Kronecker}.
The difference of size of the squares is exactly given by the defects $\nildef[j]$, $\beta\ored[j]$ and $\beta\cred[j]$.
The dark squares bearing the number $j$ represent all the nilpotent blocks $\mat{N}_j$; there are $\nildef[j]$ such blocks.
The light squares in the lower-right part bearing the number $j$ represent the $\mat{L}$-blocks $\mat{L}_j$.
There are $\beta\ored[j]$ such blocks.
The light squares in the upper-left part bearing the number $j$ represent the $\mat{L}$-blocks $\mat{L}_j^{\transpose}$.
There are $\beta\cred[j]$ such blocks.
This figure also allows to check the formulae of \autoref{prop_rel_defect_dims}.
}

% \begin{lemma}
% \[\begin{aligned}
% N\red{k}  &=
% W\red{k} &=
% \end{aligned}\]
% \end{lemma}

% \begin{proof}
% The proof is essentially already contained in the proof of \autoref{thm_linear_decomposition}
% \end{proof}

\subsection{Weak Equivalence and Strangeness}
\label{sec_strangeness}

In order to illustrate the power of the reduction point of view, and to show an application of \autoref{lma_coupling_system}, we show a normal form for the weak equivalence relation.
% Instead of looking at the equivalence classes for the equivalence of matrices, that is, pairs of invertible operators acting on $(\E,\A)$ as $(\Pm\E\Qm,\Pm\A\Qm)$, we look at the \emph{weak} equivalence.

\newcommand*\Gw{\mathsf{G}}
\newcommand*\Gr{\mathsf{H}}

Define the set consisting of block matrices
\begin{align}
(\Qm,\Rm)\coloneqq \begin{bmatrix}
	\Qm & \Rm \\
	& \Qm
\end{bmatrix}
,
\qquad
\Qm \in \GL[n]
\quad
\Rm \in \RR^{n\times n}
\end{align}
This set forms a group that we denote $\Gr$.
The inverse of an element $(\Qm,\Rm)$ is given by $(\Qm\inv, -\Qm\inv \Rm\Qm\inv)$.
The subset of group elements of the form $(\one,\Rm)$ is a commutative, normal subgroup of $\Gr$, which is isomorphic to the vector space $\RR^{n\times n}$ considered as a commutative group.
Finally, elements of $\Gr$ can be decomposed as a product of elements of $\GL[n]$ and $\RR^{n\times n}$ by:
\begin{align}
	\label{eq:semidirect}
(\Qm,\Rm) = (\Qm,0) \mul (\one,\Qm\mul\inv\Rm) = (\one, \Rm\mul\Qm\inv)\mul (\Qm,0)
\end{align}

Define the group $\Gw\coloneqq \GL[m] \times \Gr$.
An element $(\Pm,(\Qm,\Rm))$ acts on a system $(\E,\A)$ by matrix multiplication as
\begin{align}
	(\Pm,(\Qm,\Rm)) \cdot (\E,\A) \coloneqq \Pm\inv \mul \begin{bmatrix}
		\E & \A
	\end{bmatrix}
	\mul
	(\Qm,\Rm)
\end{align}
The explicit formula is
\begin{align}
	\label{weak_equiv_action_eq}
% (\Pm,(\Qm,\Rm)) \cdot (\E,\A) = (\Pm\E\Qm\inv, \Pm\mul\paren{-\E\Qm\inv\Rm + \A}\mul\Qm\inv)
(\Pm,(\Qm,\Rm)) \cdot (\E,\A) = (\Pm\inv\E\Qm, \Pm\inv\mul\paren{\E\mul\Rm + \A\mul\Qm})
,
\end{align}
and agrees with that of~\cite[eq.~VII.1.8]{HaWa96}.

We say that two systems $(\E,\A)$ and $(\E', \A')$ are \demph{weakly equivalent} if they are on the same $\Gw$-orbit.
For the study of the orbits of the weak equivalence group,~\eqref{eq:semidirect} gives
\begin{equation}
\label{orbit_sep_eq}
(\Pm,(\Qm,\Rm)) = (\Pm,(\Qm,0))\mul (\one,(\one,\Rm \Qm\inv)) = (\one,(\one,\Qm\inv \Rm)) \mul (\Pm,(\Qm,0))
\end{equation}
so we can restrict our attention to one subgroup at a time.

We now proceed to define a normal form that settles the question of weak equivalence.

% \subsubsection{Orbit Invariants}

The orbits of the weak equivalence group action~\eqref{weak_equiv_action_eq} were studied in~\cite{Mehrmann}, in which the authors exhibited a complete set of invariants.
We give an alternative proof here, thereby shedding some light on the notion of \emph{strangeness}.

\begin{theorem}
\label{thm_strangeness}
	A complete set of invariants for the group action~\eqref{weak_equiv_action_eq} is given by
	\begin{enumerate}
		\item $d \coloneqq \dim \Vv\ored[2]$,
		%\item $u \coloneqq \dim \coker [\A]$,
		\item $a \coloneqq  \dim \coim \prline$
			% $\nildef[1]$,
		% \item $s \coloneqq \dim \Vv\ored[1] - d$.
		\item $s \coloneqq \dim \Vv\ored[1] / \Vv\ored[2]$.
	\end{enumerate}
\end{theorem}

The integer $s$ is called ``strangeness'' in~\cite{Mehrmann}.

\begin{proof}
\begin{enumerate}
\item
First we have to check that the three integers are indeed invariants of the group action.
Clearly, they are invariants by transformations of the form $(\Pm,(\Qm,0))$, which are merely equivalent transformation.

Let us examine the case of a transformation
\begin{align}
(\Eb,\Ab) \coloneqq (\one,\one,\Rm)\cdot (\E,\A) = (\E, \E \Rm + \A)
.
\end{align}
First, notice that
\(
	{\vv}\ored = \im\Eb  = \im\E = \Vv\ored
\).
Now, notice that the projection of $\A$ on $\coker\E$ is the same as the projection of $\Ab$ onto $\coker\Eb$, as $\E\Rm$ projects to zero, so $\pline[\Ab] = \pline$.
This gives $\mv\ored = \Mv\ored$.
We conclude that
\begin{align}
(\Eb\ored,\Ab\ored) = (\E\ored,\A\ored)
.
\end{align}
In particular, this shows that
\(
	\vv\ored[2] = \im\Eb\ored = \im\E\ored = \Vv\ored[2]
\).
Using~\eqref{orbit_sep_eq}, this shows that the spaces $\Vv\ored$, $\Vv\ored[2]$ and $\Mv\ored$ are invariants of all of the weak equivalence group transformations.
Finally, $\prline[\Ab]$ is the restriction of $\pline[\Ab] = \pline$ on $\ker\Eb = \ker\E$, so $\prline[\Ab] = \prline$.
This shows that the defect $\nildef[1]$ is invariant.

We conclude that that $\dim\Vv\ored$, $\dim\Vv\ored[2]$ and $\nildef[1]$ are invariant under the group action.

\item
Now we show that the integers $d$, $a$ and $s$ are the only invariants.
In order to show that, we show that a system $(\E,\A)$ is weakly equivalent to a canonical form that depends only on those three integers.
%This is similar to what we will do in \autoref{sec_coupling}.

In order to achieve this, we decompose $\Mv$ and $\Vv$ using \autoref{lma_coupling_system}.
\begin{enumerate}
\item
Let us choose an arbitrary decomposition
\[
\Vv\ored = \Vv\ored[2] \oplus \sect{\coker\E\ored}
.
\]
We may now apply \autoref{lma_coupling_system} to obtain spaces $\sect{\coim\pline[E]}$, $\sect{\coim\prline}$, $\sect{\im\pline}$ and $\sect{\coker\pline}$ equipped with appropriate bases.
% Using \autoref{prop_dimKZ_ab} we obtain $s = \dim \Wv''$ and $a = \dim \Kv'$.
\item
Finally, define $\Pi$ as a projector from $\Mv$ to $\Mv\ored$ along $\sect{\coim\pline}$. 
Let $\mat{K}$ be a right inverse for $\E$ on $\Vv\ored = \im\E$. 

Define
\[\Rm := -\mat{K}\A \Pi
,
\]
so $\E\Rm + \A = 0$ on $\Mv\ored$.

As a result, if we define the new system $(\Eb,\Ab)$ by
\[(\Eb,\Ab) := (\E,\E\Rm+\A)
,
\]
then the restriction of $\Ab$ on $\Mv\ored$ is zero.
\item
Now we can choose a basis of $\Mv\ored$ and of $\Vv\ored[2]=\im\E\ored$ such that $\Eb$ is represented by the identity matrix on $\Mv\ored$. 

This provides us with complete basis of $\Mv$ and $\Vv$ such that the matrices $\Eb$ and $\Ab$ take the form described in \autoref{fig_strangeness}.
\end{enumerate}
\end{enumerate}
\end{proof}

\begin{figure}
	\centering
\includegraphics[width=.5\textwidth]{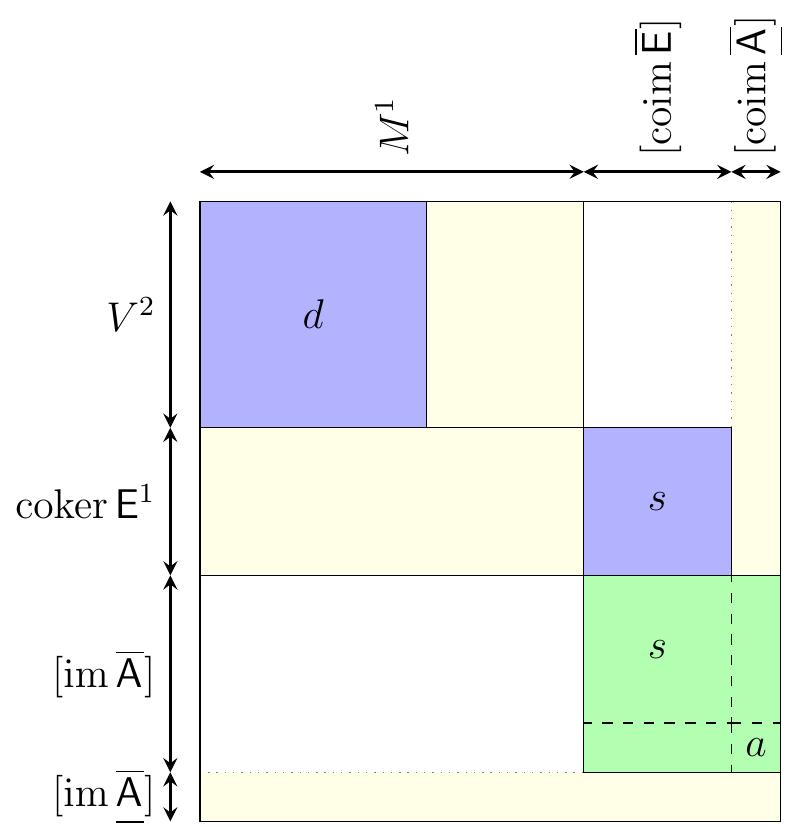}
\caption{Canonical form of a matrix corresponding to the weak equivalence. The matrix $\E$ is represented in blue, whereas the matrix for $\A$ is represented in green. All such squares are identity matrices. The rest is filled out by zero entries.}
\label{fig_strangeness}
\end{figure}

\appendix

\section{Appendix}
\label{sec:appendix}
\begin{align}
	\label{diag:betacred}
\begin{tikzpicture}[ineq]
\matrix(m) [commdiag, ]
{
	\& 0 \& 0 \& 0 \&\\
	 0 \& \diagdim{\ker\rline}[\beta\cred]  \& \ker\A  \&  \ker\A\cred \& 0\\
	0 \& \ker\E \& \Uv \& \Uv\cred \& 0\\
0 \& \im\rline \& \W \& \W\cred \& 0 \\
  \& 0 \& \coker\A \& \coker\A\cred \& 0\\
 \&   \& 0 \& 0 \&\\
};
\path[->,
% amph
]
(m-2-1) edge (m-2-2)
(m-2-2) edge (m-2-3)
(m-2-3) edge (m-2-4)
(m-2-4) edge (m-2-5)
;
\path[->]
(m-3-1) edge (m-3-2)
(m-3-2) edge (m-3-3)
(m-3-3) edge (m-3-4)
(m-3-4) edge (m-3-5)
;
\path[->]
(m-4-1) edge (m-4-2)
(m-4-2) edge (m-4-3)
(m-4-3) edge (m-4-4)
(m-4-4) edge (m-4-5)
;
\path[->,
% amph
]
% (m-5-1) edge (m-5-2)
(m-5-2) edge (m-5-3)
(m-5-3) edge (m-5-4)
(m-5-4) edge (m-5-5)
;
\path[->]
(m-1-2) edge (m-2-2)
(m-2-2) edge (m-3-2)
(m-3-2) edge node[auto] {$\pline$} (m-4-2)
(m-4-2) edge (m-5-2)
% (m-5-2) edge (m-6-2)
;
\path[->]
(m-1-3) edge (m-2-3)
(m-2-3) edge (m-3-3)
(m-3-3) edge node[auto] {$\A$} (m-4-3)
(m-4-3) edge (m-5-3)
(m-5-3) edge (m-6-3)
;
\path[->,
% amph
]
(m-1-4) edge (m-2-4)
(m-2-4) edge (m-3-4)
(m-3-4) edge node[auto] {$\A\cred$} (m-4-4)
(m-4-4) edge (m-5-4)
(m-5-4) edge (m-6-4)
;
\end{tikzpicture}
\end{align}
\begin{align}
	\label{diag:alphacred}
\ninediag%
{\ker\E\cred \& \Mv\cred \& \Mv\cred[2]}%
{\im\rline \& \Vv \& \Vv\cred}%
{\diagdim{\im\prline}[\nildef[1]] \& \coker\E \& \coker\E\cred}%
[\path (m-2-2) edge node[auto]{$\rline[\E]$} (m-3-2)
(m-2-3) edge node[auto] {$\E$} (m-3-3)
(m-2-4) edge node[auto] {$\E\cred$} (m-3-4)
;]
	% \begin{tikzpicture}[ineq]
	% 	\matrix(m)[commdiag]
	% 	{ \& 0 \& 0 \& 0 \& \\
	% 	0 \& \ker\E\cred[1] \& \ker\E \& \coim\prline\diagdim{\nildef[1]} \& 0\\
	% 	0 \& \Mv\ored \& \Mv \& \coim\pline \& 0 \\
	% 	0 \& \Vv\ored[2] \& \Vv\ored \& \coker\E\ored \& 0 \\
	% 	 \& 0 \& 0 \& 0 \&  \\};
	% 	\path[->]
	% 	(m-1-2) edge (m-2-2)
	% 	(m-2-2) edge (m-3-2)
	% 	(m-3-2) edge node[auto] {$\E\ored$} (m-4-2)
	% 	(m-4-2) edge (m-5-2)
	% 	;
	% 	\path[->]
	% 	(m-1-3) edge (m-2-3)
	% 	(m-2-3) edge (m-3-3)
	% 	(m-3-3) edge node[auto] {$\E$} (m-4-3)
	% 	(m-4-3) edge (m-5-3)
	% 	;
	% 	\path[->]
	% 	(m-1-4) edge[amph] (m-2-4)
	% 	(m-2-4) edge[amph] (m-3-4)
	% 	(m-3-4) edge[amph] (m-4-4)
	% 	(m-4-4) edge[amph] (m-5-4)
	% 	;
	% 	% \path[->]
	% 	% (m-2-5) edge (m-3-5)
	% 	% (m-3-5) edge (m-4-5)
	% 	% (m-4-5) edge (m-5-5)
	% 	% ;
	% 	\path[->]
	% 	(m-2-1) edge[amph] (m-2-2)
	% 	(m-2-2) edge[amph] (m-2-3)
	% 	(m-2-3) edge[amph]  (m-2-4)
	% 	(m-2-4) edge[amph] (m-2-5)
	% 	;
	% 	\path[->]
	% 	(m-3-1) edge (m-3-2)
	% 	(m-3-2) edge (m-3-3)
	% 	(m-3-3) edge  (m-3-4)
	% 	(m-3-4) edge (m-3-5)
	% 	;
	% 	\path[->]
	% 	(m-4-1) edge (m-4-2)
	% 	(m-4-2) edge (m-4-3)
	% 	(m-4-3) edge (m-4-4)
	% 	(m-4-4) edge (m-4-5)
	% 	;
	% \end{tikzpicture}
\end{align}

\subsection*{Acknowledgement}
We wish to acknowledge the support of the \href{http://wiki.math.ntnu.no/genuin}{GeNuIn Project}, funded by the Research Council of Norway, and that of its supervisor, Elena Celledoni, as well as the support of the J~C~Kempe memorial fund (grant no SMK-1238).

\newcommand{\includebibliography}[1]{
	\bibliographystyle{abbrvurl}
\bibliography{../#1}
}
\includebibliography{IDE}

\end{document}